\documentclass{amsart}
\usepackage{lipsum}
\usepackage{amsfonts}
\usepackage{graphicx}
\usepackage{epstopdf}
\usepackage{algorithmic}
\usepackage{hyperref}
\usepackage{float}
\usepackage[caption=false]{subfig}
\usepackage{amssymb}
\usepackage{fancyhdr}

\newcommand{\bn}{\boldsymbol{n}}

\newcommand{\dx}{\,dx}

\newcommand{\ds}{\,ds}

\newcommand{\Rmnum}[1]{\expandafter\@slowromancap\romannumeral #1@}

\newcommand{\R}{\mathbb{R}}
\newcommand{\bC}{\mathbb{C}}

\newcommand{\bA}{\boldsymbol{A}}
\newcommand{\bB}{\boldsymbol{B}}
\newcommand{\bM}{\boldsymbol{M}}
\newcommand{\bH}{\boldsymbol{H}}

\newcommand{\bt}{\boldsymbol{t}}

\newcommand{\bAt}{\boldsymbol{A}}

\newcommand{\et}{e}

\newtheorem{assumption}{Assumption}
\ifpdf
  \DeclareGraphicsExtensions{.eps,.pdf,.png,.jpg}
\else
  \DeclareGraphicsExtensions{.eps}
\fi

\newcommand{\yemeifont}{\fontsize{9pt}{\baselineskip}\selectfont}
\newtheorem{remark}{Remark}[section]

\newtheorem{theorem}{Theorem}[section]
\newtheorem{lemma}[theorem]{Lemma}

\theoremstyle{definition}

\title{An energy stable and maximum bound principle preserving scheme for the dynamic Ginzburg--Landau equations under the temporal gauge}
\author{Limin Ma, Zhonghua Qiao}

\address{School of Mathematics and Statistics, Wuhan University, Wuhan, Hubei 430072, China, limin18@whu.edu.cn}
\address{Department of Applied Mathematics, The Hong Kong Polytechnic University, Hung Hom, Kowloon, Hong Kong, zhonghua.qiao@polyu.edu.hk}
\date{}							

\begin{document}
\maketitle 

\begin{abstract}
This paper proposes a decoupled numerical scheme of the time-dependent Ginzburg--Landau equations under the temporal gauge. For the magnetic potential and the order parameter, the discrete scheme adopts the second type Ned${\rm \acute{e}}$lec element and the linear element for spatial discretization, respectively; and a linearized backward Euler method and the first order exponential time differencing method for time discretization, respectively. The maximum bound principle (MBP) of the order parameter and the energy dissipation law in the discrete sense are proved. The discrete energy stability and MBP-preservation can guarantee the stability and validity of the numerical simulations, and further facilitate the adoption of an adaptive time-stepping strategy, which often plays an important role in long-time simulations of
vortex dynamics, especially when the applied magnetic field is strong.  An optimal error estimate of the proposed scheme is also given.
Numerical examples verify the theoretical results of the proposed scheme and demonstrate the vortex motions of superconductors in an external magnetic field.

  \vskip 15pt

 \vskip 15pt


\end{abstract}

\section{Introduction}
In this paper, we consider the transient behavior and vortex motions of superconductors in an external magnetic field $\bH$ which is described by the time-dependent Ginzburg--Landau (TDGL) model \cite{ginzburg1950theory}. This model was first established in \cite{gor1996generalization} with some detailed descriptions in \cite{chapman1992macroscopic,du1992analysis,tinkham2004introduction}. The TDGL equations in the non-dimensional form satisfy
\begin{equation}\label{model0}\small
\left\{
\begin{aligned}
\left(\partial_{t}+ \mbox{i}\kappa  \phi\right) \psi
+ \left(\frac{ \mbox{i}}{\kappa} \nabla + \boldsymbol{A}\right)^{2} \psi
+ (|\psi|^{2} -1) \psi &= 0& \quad \mbox{in }\ \Omega\times (0,T],\\
\sigma\left(\nabla \phi + \partial_{t} \bA\right)
+ \nabla \times(\nabla \times \boldsymbol{A})
+ Re\left[  \psi^*(\frac{\mbox{i}}{\kappa}\nabla + \bA)\psi  \right]
&=\nabla\times \bH& \quad \mbox{in }\ \Omega\times (0,T],
\end{aligned}
\right.
\end{equation}
with boundary and initial conditions
\begin{equation}\label{model0bc}\small
\left\{
\begin{aligned}
(\nabla \times \boldsymbol{A})\times \boldsymbol{n}&=\boldsymbol{H}\times \boldsymbol{n},
&\quad (\frac{ \mbox{i}}{\kappa} \nabla + \boldsymbol{A}) \psi\cdot \boldsymbol{n} &=0&\text { on } \partial \Omega,
\\
\psi(x,0) &=\psi^0(x),&\quad \bA(x,0)&=\bA^0(x)&\mbox{ on } \Omega,
\end{aligned}
\right.
\end{equation}
where $\Omega$ is a bounded domain in $\mathbb{R}^d(d=2, 3)$, $\bn$ is the unit outer normal vector, the electric potential $\phi$ is a real scalar-valued function, the Ginzburg-Landau parameter $\kappa$ is an important positive material constant representing the ratio of penetration length to the coherence length, the relaxation parameter $\sigma$ is a given positive constant, the magnetic potential~$\bA$ is a real vector-valued function and the order parameter~$\psi$ is a complex scalar-valued function.
Physically speaking, the magnitude of the order parameter $|\psi|$ represents the superconducting density, where $|\psi|=0$ stands for the normal state,  $|\psi|=1$ for the superconducting state, and $0<|\psi|<1$ for a mixed state.
It is proved in \cite{chen1993non} that the order parameter in the TDGL equations \eqref{model0} satisfies the MBP in the sense that the magnitude of the order parameter is bounded by 1, i.e.
\begin{equation}\label{MBPexact}
\|\psi(\cdot, t)\|_{\infty}\le 1,\quad \forall\ t>0
\end{equation}
if the initial condition $\|\psi^0\|_{\infty}\le 1$.
The solution of the corresponding stationary Ginzburg--Landau equations minimizes the Gibbs energy functional \cite{kopnin,etienne2008vortices}
\begin{equation}\label{energy:def}
G(\bA,\psi)=\frac12\|(\frac{\mbox{i}}{\kappa}\nabla+\bA)\psi\|_0^2 + \frac12\|\nabla\times \bA-\bH\|_0^2 + \frac14\||\psi|^2-1\|_0^2.
\end{equation}
As analyzed in \cite{mu1998an},  the energy dissipation law below holds for \eqref{model0}
\begin{equation}\label{energyexact}
\frac{d}{dt}G(\bA,\psi) \le -4\pi(\bM, \partial_t\bH),
\end{equation}
where the magnetization $\bM=\frac{1}{4\pi}(\nabla\times \bA - \bH)$.
Particularly, if the applied magnetic field $\bH$ is stationary, the Gibbs energy of a solution of \eqref{model0} decreases in time.
As stated in \cite{du1994global}, the solution of \eqref{model0} is not unique, that is given any solution $(\psi, \bA, \phi)$, a gauge transformation
$
G_\chi (\psi, \bA, \phi)=(\psi e^{i\kappa \chi}, \bA + \nabla \chi, \phi - \partial_t \chi)
$
gives a class of equivalent solutions sharing the same $|\psi|$ and magnetic induction field $\nabla\times \bA$, which are of physical interests. Although the solutions of \eqref{model0} under different gauges are theoretically equivalent, numerical schemes under various gauges are computationally different. 
The temporal gauge is adopted in the paper since the corresponding TDGL equations can be viewed as a gradient flow and admits the energy dissipation property when $\bH$ is stationary. The existence and uniqueness of the TDGL equations \eqref{model0}-\eqref{model0bc}  were given in \cite{chen1993non,du1994global,li2017mathematical}.

For the TDGL equations, some numerical schemes using finite difference methods for spatial discretization were proposed and analyzed to preserve the discrete MBP and energy bound in \cite{du1998discrete,du2005approximation,gao2019a}. These MBP-preserving finite difference schemes require {uniform or rectangular meshes, and the bound of the discrete energy may be very large in long-time simulations. Numerical schemes using finite element methods for spatial discretization can simulate the motion of superconductors with more general shapes, and are easy to be extended to three-dimensional simulations.
Many finite element based numerical schemes were proposed and analyzed for different gauges, especially the temporal gauge $\phi=0$ (see e.g., \cite{du1994finite,mu1997linearized,mu1998an}) and the Lorentz gauge $\phi=-\nabla\cdot \bA$ (see e.g., \cite{chen1997mixed,gao2014optimal,gao2018analysis,li2017convergence}) under an additional boundary condition. This boundary condition is indispensable to guarantee the wellposedness of the discrete problems and analyze the convergence rate of numerical solutions. However, the regularity of the finite element solution under such boundary conditions is higher than expected, which leads to some nonphysical phenomena if the mesh is not refined enough. Two mixed finite element methods using Hodge decomposition in \cite{li2020a,li2015new} weakly impose this additional boundary condition on the approximation of $\bA$ for the TDGL equations under the Lorentz gauge, which avoid the nonphysical phenomenon to a certain extent  for the TDGL equations in nonconvex polygons.
Recently, a nonlinear numerical scheme with no additional boundary condition was proposed in \cite{duan2022residual,hong2022efficient} for the TDGL equations under the temporal gauge, which resolves physical-interested phenomena on relative coarse meshes.
The energy dissipation law was proved under a strict restriction on time steps in \cite{hong2022efficient}. But no MBP analysis was provided for this scheme.

It is of great importance to analyze the {  MBP~\eqref{MBPexact} and energy dissipation law \eqref{energyexact}} 
for these finite element based schemes in the literature.
{Although the discrete MBP for the TDGL equations is usually observed for finite element based schemes, it has not been proved theoretically.
The magnitude of the discrete order parameter was proved to be bounded above in \cite{mu1998an} under the assumption $\tau\lesssim h^{11\over12}$ and $\tau\lesssim h^2$ in two and three dimensions, respectively.
}
The TDGL equations under the Lorentz gauge cannot be viewed as a gradient flow of the Gibbs energy, and thus the energy stability analysis of numerical schemes concerning this gauge is difficult and the relevant work is very limited in the literature. The boundedness of a modified energy with an extra term $\frac12\|\psi\|_0^2$ was analyzed for the scheme in \cite{li2017mathematical} concerning the Lorentz gauge with the bound depending on the terminal time.
The TDGL equations under the temporal gauge can be viewed as an~$L^2$-gradient flow with respect to $G(\bA,\psi)$ and
	\begin{equation}\label{energyid}
		\frac{d}{dt}G(\bA,\psi) + \|\partial_t \bA\|_0^2 + \|\partial_t\psi\|_0^2=-4\pi(\bM, \partial_t\bH),
	\end{equation}
which benefits the energy stability analysis of numerical schemes under this particular gauge. The discrete energy dissipation law was analyzed for the nonlinear schemes in \cite{du1994finite,hong2022efficient}, where the uniqueness of solution for both schemes requires time step sizes $\tau\lesssim h^{d/2}$ where $d$ is the dimension of space. A modified energy was proved to be bounded in \cite{mu1998an}, where the bound tends to infinity as the perturbed model tends to the original one.

In this paper, we propose a decoupled numerical scheme for the TDGL equations under the temporal gauge
\begin{equation}\label{model0temporal}
\left\{
\begin{aligned}
\partial_{t}\psi
+ \left(\frac{ \mbox{i}}{\kappa} \nabla + \boldsymbol{A}\right)^{2} \psi
+ (|\psi|^{2} -1) \psi &= 0& \quad \mbox{in }\ \Omega\times (0,T]
\\
\sigma \partial_{t} \bA
+ \nabla \times(\nabla \times \boldsymbol{A})
+ Re\left[\psi^*(\frac{\mbox{i}}{\kappa}\nabla + \bA)\psi  \right]
&=\nabla\times \bH& \quad \mbox{in }\ \Omega\times (0,T]
\end{aligned}
\right.
\end{equation}
with boundary and initial conditions \eqref{model0bc}. The scheme employs the lowest order second type Ned${\rm \acute{e}}$lec element and the  linear Lagrange element with mass lumping for finite element discretization of $\bA$ and~$\psi$ in space, respectively. For time discretization, the proposed scheme solves $\bA$ first by the backward Euler method with the nonlinear term treated explicitly, and then $\psi$ by the first order exponential time differencing (ETD) method \cite{beylkin1998new,cox2002exponential,hochbruck2005explicit,hochbruck2010exponential}. The ETD method has been proved to preserve the discrete MBP in many applications, see e.g., \cite{du2019maximum,du2021maximum,ju2022generalized,li2020arbitrarily}. Different from the MBP analysis for real-valued differential equations, the complexity of the order parameter $\psi$ leads to a complex-valued matrix that is not diagonally dominant, and poses difficulty in the MBP analysis for \eqref{model0temporal}. Besides, the highly coupled terms in \eqref{model0temporal} add to the difficulty in analyzing the energy dissipation and error estimate for the proposed decoupled scheme. For the proposed decoupled scheme, we analyze the discrete MBP-preserving property and the discrete energy dissipation law with respect to the original Gibbs energy, and give an optimal error estimate. This is the first finite element based scheme that preserves the strict discrete MBP \eqref{MBPexact} theoretically, and the first decoupled finite element based scheme that admits the discrete energy dissipation law \eqref{energyexact} with respect to the original energy~\eqref{energy:def}.
These stabilities are of great benefit since they allow the application of adaptive time-stepping strategy in \cite{qiao2011adaptive} to significantly speed up long-time simulations.

The rest of the paper is organized as follows. The decoupled numerical scheme is presented in Section~\ref{sec:scheme}.
{ The discrete MBP for the order parameter and an unconditional energy stability are analyzed  in Section~\ref{sec:MBP} and Section~\ref{sec:energy}, respectively.}
The error estimate of the numerical scheme is given in Section~\ref{sec:estimate}. Some numerical experiments are carried out in Section~\ref{sec:numerical} to verify the theoretical results and demonstrate the performance of the proposed scheme in long-time simulations. The paper ends with some concluding remarks in Section~\ref{sec:conclusions}.

\section{Fully discrete scheme for the TDGL equations}\label{sec:scheme}
In this section, we present the fully discrete scheme for \eqref{model0temporal}. Some standard notations are given below. Let $\bC$ be the set of complex numbers, $L^2(\Omega,\mathbb{R})$, and $H^1(\Omega,\mathbb{R})$ be the conventional Sobolev spaces defined on a domain $\Omega\subset \mathbb{R}^d$ ($d=2$ or $3$). For any two complex functions $v$, $w\in L^2(\Omega, \bC)$, denote the $L^2(\Omega, \bC)$ inner product and the norm by
$
(v,w)=\int_\Omega vw^*\dx$, $\|v\|_0^2=\int_\Omega |v|^2\dx,
$
respectively, where $w^*$ is the conjugate of $w$ and $|v|$ is the magnitude of $v$. Denote the complex-valued Sobolev space as
$$
H^1(\Omega, \mathbb{C})=\{\phi=u+\mbox{i}v: u, v\in H^1(\Omega, \mathbb{R})\},
$$
and the vector-valued space with $d$ components as
$$
H({\rm curl})=\{\bB: \bB\in L^2(\Omega,\mathbb{R}^d),\ \nabla\times \bB\in L^2(\Omega,\mathbb{R}^d)\}.
$$
The weak formulation of the TDGL equations \eqref{model0temporal} with boundary conditions \eqref{model0bc} is specified as follows: find $(\bA, \psi)\in H({\rm curl})\times H^1(\Omega, \mathbb{C})$ such that
\begin{equation}\label{weakform}
\left\{
\begin{aligned}
&(\sigma\partial_t \bA,\bB) + D(\psi; \bA, \bB) + (g(\psi), \bB)=(\bH, \nabla\times \bB),& \forall \bB\in H({\rm curl}),\\
&(\partial_t \psi,\phi) + B(\bA;\psi, \phi) - (f_0(\psi), \phi) = 0,&\forall \phi\in H^1(\Omega,\mathbb{C}),
\end{aligned}
\right.
\end{equation}
with $\bA(x, 0)=\bA^0(x)\in H({\rm curl})$ and $\psi(x,0)=\psi^0(x)\in H^1(\Omega,\mathbb{C})$, where
\begin{equation}
\begin{aligned}
D(\psi; \bA, \bB)&=(\nabla\times \bA, \nabla\times \bB)+ (|\psi|^2\bA, \bB),&
g(\psi)&=\frac{\mbox{i}}{2\kappa} (\psi^*\nabla \psi - \psi\nabla \psi^*),
\\
B(\bA;\psi, \phi)&=((\frac{\mbox{i}}{\kappa}\nabla + \bA)\psi, (\frac{\mbox{i}}{\kappa}\nabla + \bA)\phi),&
f_\mu(x)&=(1-|x|^2)x+\mu x
.\label{Bdef}
\end{aligned}
\end{equation}

Let $\mathcal{T}_h$ be a regular partition of $\Omega$, $\mathcal{E}_h$ be the set of all interior edges of $\mathcal{T}_h$, $\bt_e$ be the unit tangent vector of an edge $e\in \mathcal{E}_h$ and $h_K$ be the diameter of element $K\in \mathcal{T}_h$. Define the mesh size $h=\max_{K\in \mathcal{T}_h}h_K$. Let $P_1(K,\mathbb{C})$ be the set of all polynomials with degree not greater than one. Define the linear element space by
\begin{equation*}
V_h=\{\phi_h\in H^1(\Omega, \mathbb{C})\cap C^0(\Omega,\mathbb{C}): \phi_h|_K \in P_1(K,\mathbb{C})\},
\end{equation*}
and the lowest order second type Ned${\rm \acute{e}}$lec element space by
\begin{equation*}
Q_h=\{ \bB_h\in H(\mbox{curl}):  \bB_h|_K\in P_1(K, \mathbb{R}) , \ \int_e \bB_h\cdot \bt_e\,ds \ \mbox{ is continuous on any }\ e\in \mathcal{E}_h\}.
\end{equation*}
Let $\Pi_L$ be the canonical interpolation operator of the linear element, namely
$\Pi_L v(x)=\sum_{i=1}^Nv(x_i)\phi_i(x)$, where $N$ is the number of vertices $\{x_i\}_{i=1}^N$ of $\mathcal{T}_h$, and $\phi_i\in V_h$ is the corresponding basis function with respect to vertex $x_i$ with
$
\phi_i(x_j)=\delta_{ij}.
$
 Let $\omega_i$ be the support of $\phi_i(x)$.
Define a diagonal matrix $D={\rm diag}(d_1,\cdots, d_N)$ with entries $d_i=|\phi_i|_{0,1,\omega_i}$.
Denote the inner product
$
(V,W)_{\ell^2}=W^HDV=\sum_{i=1}^{N}V_iW_i^*|\phi_i|_{0,1,\omega_i}$ for any $V$, $W\in \bC^{N}$,
and the operators $I_h: V_h\rightarrow \bC^N$ and $\Pi_h: \bC^N\rightarrow V_h$ by
$
I_h w=(w(x_1),\cdots,w(x_N))^T$ and $
\Pi_hW=\sum_{i=1}^NW_i\phi_i(x),
$
respectively.
Note that
\begin{equation}\label{L2def}
(I_hv, I_hw)_{\ell^2}=(\Pi_L(vw^*), 1),\quad \|v\|_0\lesssim \|I_hv\|_{\ell^2}\lesssim \|v\|_0,
\end{equation}
where the notation $A\lesssim B$ means that there exists a positive constant $C$, which is independent of the mesh size, such that $A\le CB$.
Define the Ritz projection $R_{h}\bA\in Q_h$ by
\begin{equation}\label{ritz}
(\nabla \times (\bA - R_{h}\bA), \nabla \times \bB_h) + (\bA - R_{h}\bA, \bB_h)=0,\quad \forall \bB_h\in Q_h,
\end{equation}
which admits the following estimates on a convex domain \cite{Peter2003FiniteEM}:
\begin{equation}\label{ritz:err}
\|\nabla\times (I-R_{h})\bA\|_0 + \|(I-R_{h})\bA\|_0\lesssim h(|\bA|_1 + |\nabla\times \bA |_1),
\end{equation}
provided that $\bA, \nabla\times \bA\in H^1(\Omega, \mathbb{R}^d)$ and 
\begin{equation}\label{ritz:err2}
h\|\nabla\times (I-R_{h})\bA\|_0 + \|(I-R_{h})\bA\|_0\lesssim h^2|\bA|_2,
\end{equation}
provided that $\bA\in H^2(\Omega, \mathbb{R}^d)$.
Given a positive integer $K_t$ and time steps $\{\tau_i\}_{i=1}^{K_t}$, we divide the time interval by $\{t_n=\sum_{i=0}^n\tau_i: 0\le n\le K_t\}$ and $T=t_{K_t}$. For any function $F(\cdot,t)$, define $F^n=F(\cdot, t_n)$ and $\partial_t^n F=\partial_t F(\cdot, t_n)$. For any given sequence of functions $\{F^n\}$, denote
$
 d_t^n F= ({F^n-F^{n-1})/\tau_n}.
$

Let  $\bA_h^0 = R_{h}\bA^0$  and $\Psi_h^0 = I_h\psi^0$.
Given the approximation $(\bA_h^{n-1}, \Psi_h^{n-1})\in Q_h\times \mathbb{C}^N$ at the previous time step $t_{n-1}$, we first solve the approximation to $\bA^n$ by applying the backward Euler method for time discretization and treating the nonlinear terms explicitly. That is to find $\bA_h^n\in Q_h$ such that for any $\bB_h\in Q_h$,
\begin{equation}\label{dis1}
(d_t^n \bA_h,\bB_h) + D(\psi^{n-1}_h; \bA_h^n, \bB_h)=(\bH^n,\nabla\times \bB_h) - (g(\psi_h^{n-1}), \bB_h),
\end{equation}
where $\psi_h^{n-1}=\Pi_h\Psi_h^{n-1}$. We adopt the first order exponential time differencing method (ETD1) with stabilization for time discretization of $\psi$ and  the   linear finite element method with mass lumping for spatial discretization by  treating the nonlinear terms $B(\bA;\psi, \phi)$ and $f_0(\psi)$ in \eqref{weakform} explicitly. To be specific, we seek $u_h\in C^1([t_{n-1},t_{n}];V_h)$ such that $\psi_h^n=u_h(\cdot,t_n)\in V_h$ with $u_h(\cdot, t_{n-1})=\psi_h^{n-1}$ such that for any~$\phi_h\in V_h$ and $t\in [t_{n-1}, t_n]$,
\begin{equation*}\label{etd1lumping}
\begin{split}
(\Pi_L(\partial_t u_h\phi_h^*), 1) + B(\bA_h^{n}; u_h, \phi_h) + \mu_n(\Pi_L(u_h \phi_h^*),1)
- (\Pi_L(f_{\mu_n}(\psi_h^{n-1})\phi_h^*),1)  &= 0,
\end{split}
\end{equation*}
where $\mu_n>0$ is the stabilization parameter and $\bA_h^n$ is given by \eqref{dis1}. The matrix form of this formulation reads
\begin{equation}\label{psieq}
\left\{
\begin{aligned}
\frac{d}{dt}U_h(t)&=L_{\mu_n, h}^n U_h(t)+ f_{\mu_n} (\Psi_h^{n-1}),\qquad \forall t\in [t_{n-1}, t_n],
\\
U_h(t_{n-1})&=\Psi_h^{n-1},
\end{aligned}
\right.
\end{equation}
where $U_h(t)=I_hu_h(\cdot, t)\in \bC^N$ and the entries of the complex matrix $L_{\mu_n, h}^n$ are
\begin{equation}\label{ETD:L}
L_{\mu_n, h}^n= D^{-1}\hat L^n - \mu_n I,\quad \mbox{with}\quad (\hat L^n)_{ij} = -B(\bA_h^n; \phi_j, \phi_i).
\end{equation}
Since the diagonal matrix $D$ is positive definite and the Hermitian matrix $\hat L^n$ is negative semi-definite, $L_{\mu_n, h}^n$ is negative definite for any $\mu_n>0$, i.e.
\begin{equation}\label{negativedef}
W^*L_{\mu_n, h}^nW\le -\mu_nW^*W,\quad \forall \ W\in \bC^N.
\end{equation}
An equivalent form of \eqref{psieq} is
\begin{equation}\label{ETDsol}
\begin{split}
\Psi_h^{n}
=&\phi_0(\tau_n L_{\mu_n,h}^n)\Psi_h^{n-1} { + }\tau_n\phi_1(\tau_n L_{\mu_n,h}^n) f_{\mu_n}(\Psi_h^{n-1}),
\end{split}
\end{equation}
where
$\phi_0(a)=e^a$ and $\phi_1(a)={({ e^a-1})/a}$ for $a\neq 0$.
We use the Krylov subspace method in \cite{niesen2012algorithm} to compute the exponential integral in \eqref{ETDsol}.

\section{Discrete energy stability and maximum bound principle} \label{sec:energy0}
In this section, we will show that the proposed scheme \eqref{dis1}-\eqref{psieq} inherits the maximum bound principle~\eqref{MBPexact} and the energy dissipation law~\eqref{energyexact} at the discrete level.

\subsection{Discrete Maximum Bound Principle}\label{sec:MBP}
In this section, we consider the discrete MBP for the complex order parameter $\psi_h$ of the proposed decoupled scheme \eqref{dis1}-\eqref{psieq}. To begin with, we consider an ODE system taking the form
\begin{equation}\label{ETDmodel}
\left\{
\begin{aligned}
&\frac{du}{dt} + \mu u=Lu+N[u] \\
& u(0,x)=u^0(x)
\end{aligned}
\right.
\end{equation}
with real-valued  constant $\mu$, $L$, $N(\xi)=\mu \xi + h(\xi)$.
An analytical framework was established in \cite{du2021maximum} to give some sufficient conditions that lead to the MBP for \eqref{ETDmodel}. This framework can be extended to complex-valued systems, which is presented below.
\begin{lemma}\label{lm:complex}
Given any  real-valued positive constant $\mu$ and $T$, assume that
\begin{enumerate}
\item[(a)] for any $U\in \mathbb{C}^N$, it holds that ${\rm Re}(U^*_i(LU)_i)< 0$ if $|U_i|=\max_{1\le j\le N}|U_j|$;
\item[(b)] there exists $\lambda_0>0$ such that $\lambda_0 I-L$ is reversible;
\item[(c)] $|N(\xi)|\le \mu \beta$ for any $|\xi|\le \beta$ and $|N(\xi_1)-N(\xi_2)|\le 2\mu |\xi_1-\xi_2|$ for any $|\xi_1|\le \beta$ and $|\xi_2|\le \beta$.
\end{enumerate}
If $\|u^0\|_{L^\infty}\le \beta$ and $\mu\ge \max_{|\xi|\le \beta}|h'(\xi)|$,  it satisfies
$\|u(t)\|_{L^\infty}\le \beta$ for any $t \in [0, T]$.
\end{lemma}
Assumptions (a) and (b) in Lemma \ref{lm:complex} indicate that the linear operator ${ L}$ 
is a generator of a contraction semigroup since assumption (a) implies
\begin{equation*} \label{contraction}
\begin{aligned}
\|(\lambda I-{ L})U\|_{\ell^\infty}^2\ge &|\lambda U_i-({ L}U)_i|^2
\\
=&\lambda^2|U_i|^2 + |({ L}U)_i|^2 -2{\rm Re}(U^*_i(LU)_i)
> \lambda \|U\|_{\ell^\infty}^2.
\end{aligned}
\end{equation*}
Lemma \ref{lm:complex} follows directly from this fact and a similar analysis in \cite{du2021maximum}. The detailed proof is omitted here.

Notice that for real-valued systems, the first assumption reduces to $U_i(LU)_i< 0$ if $|U_i|=\max_{1\le j\le N}|U_j|$ for any $U\in \R^N$, which is exactly the assumption in  \cite{du2021maximum}.
It is widely used  in the MBP analysis of ETD schemes that if all the diagonal entries of a strictly diagonally dominant matrix~$L$ are negative, assumption (a) holds for the real-valued system. For the classic two-dimensional heat equation,  since the sign of diagonal entries $L_{ij}=\int_\Omega \nabla \phi_j\cdot \nabla\phi_i\,dx$ and the corresponding $U_i(LU)_i$ depends on the interior angles, the discrete maximum principle holds for the  mass lumping method in the case that the triangulations contain no obtuse triangles \cite{nie1985a}.

Similarly, if $L$ is a Hermitian matrix with negative entries on the diagonal and strictly diagonally dominant, assumption~(a) still holds.
Although the real part of the Hermitian matrix $L_{\mu_n, h}^n$ is strictly diagonally dominant, the complex-valued off-diagonal entries make the matrix itself not even weakly diagonally dominant. If the triangulation contains some right triangles, the imaginary part of~$L_{\mu_n, h}^n$ will dominate the sign of $Re(U_i^*(L_{\mu_n, h}^nU)_i)$ when the stabilization parameter~$\mu_n$ is of $\mathcal{O}(h^{-1+\alpha})$ with $\alpha>0$. However, the sign of the imaginary part of~$L_{\mu_n, h}^n$ is uncertain, and thus the linear operator~$L_{\mu_n, h}^n$ is not necessarily the generator of a contraction semigroup on a triangulation with right interior angles. Therefore, the discrete MBP is not guaranteed.

To guarantee the discrete MBP of the complex order parameter $\psi_h$, we consider the scheme on triangulations satisfying the following assumption.
\begin{assumption}\label{ass:mesh}
The triangulation is shape regular and quasi-uniform,  where all the interior angles ($d=2$) or dihedral angles of faces ($d=3$) are acute.
\end{assumption}

By Lemma \ref{lm:complex}, the key to analyzing the discrete MBP of the solution to the ETD1 scheme \eqref{psieq} is to prove that $L_{\mu_n, h}^n$ is a generator of a contraction semigroup, namely
\begin{equation}\label{reLnegative}
Re(U_i^*(L_{\mu_n, h}^nU)_i)<0,\quad \mbox{for some }i\in \{1,\cdots,N\}
\end{equation}
holds for any $U\in \mathbb{C}^{N}$.
Note that even though the real part of the matrix $L_{0, h}^n$ is diagonally dominant, the matrix itself is not necessarily weakly diagonally dominant.
To derive the discrete MBP for the proposed scheme, we need to look into the properties of the linear operator $L_{\mu_n, h}^n$.
Denote the  entries of~$L_{\mu_n, h}^n \in \mathbb{C}^{N\times N}$ by $(L_{ij})_{i,j=1}^{N}$ with $L_{ij}=L_{ij}^{\rm re} + \mbox{i}L_{ij}^{\rm im}$ and
\begin{equation}\label{Lij}
\begin{aligned}
L_{ij}^{\rm re}&=\frac{1}{d_i}\left(-\frac{1}{\kappa^2}\int_\Omega \nabla \phi_j\cdot \nabla \phi_i\dx - \int_\Omega |\bA_h^n|^2\phi_i\phi_j\dx - \mu_n \delta_{ij}|\phi_i|_{0,1,\Omega}\right),
\\
L_{ij}^{\rm im}&=\frac{1}{\kappa d_i}\int_\Omega \bA_h^n\cdot (\phi_j\nabla \phi_i - \phi_i\nabla \phi_j)\dx,
\end{aligned}
\end{equation}
where $d_i=|\phi_i|_{0,1,\Omega}$.  It follows Assumption \ref{ass:mesh} that there exist positive constants $C_1$, $\tilde C_2$ and $C_3$, which are independent on the mesh size, such that for any $i\neq j$,
\begin{equation*}\label{meshineq}
\int_\Omega \nabla \phi_j\cdot \nabla \phi_i\dx\le -C_1h^{d-2}, \
|\kappa d_i L_{ij}^{\rm im}|\le  \tilde C_2h^{\frac12(d-2)}\|\bA_h^n\|_{0,\omega_{ij}},\
|\phi_i|_{0,1,\omega_i}\ge C_3h^d,
\end{equation*}
where the second estimate employs the Cauchy Schwarz inequality and  $\omega_{ij}=\overline{\omega_i\cap \omega_j}$ is the intersection of the support of $\phi_i$ and $\phi_j$.
For any $1\le i$, $j\le N$, define vector $\vec{v}_{ij}=(a_{ij}, b_{ij}, c_{ij})$ by
\begin{equation*}
a_{ij}=-\frac{h^{2-d}}{\kappa^2}\int_\Omega \nabla \phi_j\cdot \nabla \phi_i\dx,\quad
b_{ij}=d_i h^{1-d}L_{ij}^{\rm im},\quad
c_{ij}=\mu_n h^{-d} \int_\Omega \phi_i\phi_j\dx,
\end{equation*}
where constant $\mu_n$ is to be determined later.
It follows that each entry of the vector~$\vec{v}_{ij}$ is independent of the mesh size $h$ and
\begin{equation}\label{aest}
a_{ij}\ge \frac{C_1}{\kappa^2},\quad
0\le |b_{ij}|\le \frac{\tilde C_2}{\kappa}h^{-\frac12d}\|\bA_h^n\|_{0,\omega_{ij}},\quad
c_{ij}\ge \mu_nC_3.
\end{equation}
Assumption \ref{ass:mesh} implies that the number of elements sharing the vertices $x_i$ and $x_j$ is bounded above. Thus, there exists a positive constant $C_2$ such that
\begin{equation}\label{sumbsquare}
\sum_{j\neq i}|b_{ij}|^2<  \frac{C_2}{\kappa^2}h^{-d}\|\bA_h^n\|_{0,\omega_i}^2.
\end{equation}
For each element $K$, it holds that $\int_K\phi_i\dx=\frac13|K|$ and $\int_K\phi_i^2\dx=\frac16|K|$. Then
\begin{equation}\label{sumcsquare}
\sum_{j\neq i}c_{ij}=\mu_nh^{-d}\int_\Omega (\phi_i-\phi_i^2)\dx=\frac16\mu_nh^{-d}|\omega_i|.
\end{equation}

The following theorem shows that the operator $L_{\mu_n, h}^n$ is a generator of a contraction semigroup and the discrete MBP holds for $\psi_h$  of \eqref{ETDsol} when the stabilization parameter 
\begin{equation}\label{mucondition}
\mu_n\ge \max_{1\le i\le N}\{\frac{3C_2\|\bA_h^n\|_{0,\omega_i}^2}{C_1|\omega_i|},\frac{3\|\bA_h^n\|_{0,\omega_i}^2}{8|\omega_i|},2\},
\end{equation}
where the constants $C_1$ and $C_2$  are independent of the spatial mesh size $h$ and Ginzburg-Landau parameter $\kappa$, $\omega_i$ is the support of basis function $\phi_i$,  and $\bA_h^n$ is given by \eqref{psieq}.

\begin{theorem}\label{th:MBP1}
Assume that matrix $L_{\mu_n, h}^n$ is assembled with stabilization parameter $\mu_n$ satisfying \eqref{mucondition}  and Assumption~\ref{ass:mesh} holds. Then the discrete MBP holds for $\psi_h$  of \eqref{ETDsol}, i.e.
$$
\|\Psi_h^n\|_{\ell^\infty}\le 1,\quad \mbox{\rm if}\quad \|\psi^0\|_\infty\le 1.
$$
\end{theorem}
\begin{proof}
Define a matrix $T\in \bC^{N\times N}$ with entries $T_{ij}=U_i^*U_j = T_{ij}^{\rm re} + \mbox{i}T_{ij}^{\rm im}$. Then,
\begin{equation}\label{Rest}
\begin{aligned}
Re(U_i^*\sum_{j=1}^{N}L_{ij}U_j)
=&L_{ii}^{\rm re}T_{ii}^{\rm re} + \sum_{j\neq i} (L_{ij}^{\rm re}T_{ij}^{\rm re} - L_{ij}^{\rm im}T_{ij}^{\rm im}).
\end{aligned}
\end{equation}
It follows from $\sum_{i=1}^{N}\phi_i(x)=1$, $\sum_{i=1}^{N}\nabla \phi_i(x)=0$  and  \eqref{Lij} that
\begin{equation}\label{Liir}
L_{ii}^{\rm re}=\frac{1}{d_i}\left(\sum_{j\neq i}\int_\Omega (\frac{1}{\kappa^2}\nabla \phi_j\cdot \nabla \phi_i - \mu_n\phi_i\phi_j)\dx - \int_\Omega |\bA_h^n|^2\phi_i^2\dx - \mu_n |\phi_i|_{0,1,\Omega}\right).
\end{equation}
Substituting \eqref{Lij} and \eqref{Liir} into \eqref{Rest} yields
\begin{equation} \label{Reterms}
Re(U_i^*\sum_{j=1}^{N}L_{ij}U_j)=  \frac{1}{d_i}(R^1_i + R^2_i),
\end{equation}
where the stabilization parameter $\mu_n=\tilde\mu_1+\tilde\mu_2$ is to be determined later and
\begin{equation*}
\begin{aligned}
&R^1_i= -\sum_{j\neq i}\ell_i(U_j; \vec{v}_{ij}), \  \ell_i(U_j; \vec{v}_{ij})=a_{ij} h^{d-2}(T_{ii}^{\rm re} - T_{ij}^{\rm re}) + b_{ij}h^{d-1} T_{ij}^{\rm im} + c_{ij}h^d T_{ii}^{\rm re},
\\
& R^2_i= -\sum_{j=1}^N\int_\Omega |\bA_h^n|^2\phi_i\phi_j\,dx T_{ij}^{\rm re} - \mu_n|\phi_i|_{0,1,\Omega} T_{ii}^{\rm re}.
\end{aligned}
\end{equation*}
Let $(r_j,\theta_j)$ be the polar coordinates of $U_j$. We can find $i\in \{1,\cdots, N\}$ such that
$
r_i=\max_{1\le j\le N}r_j.
$
Then,
$
T_{ii} - T_{ij}=r_i^2 - r_ir_je^{i(\theta_j -\theta_i)}.
$
Note that
\begin{equation*}
\begin{aligned}
\ell_i(U_j; \vec{v}_{ij})&=a_{ij}r_i^2h^{d-2} - a_{ij}r_ir_jh^{d-2}\cos(\theta_j -\theta_i) - b_{ij}r_ir_jh^{d-1}\sin(\theta_j -\theta_i)+c_{ij}r_i^2h^d
\\
&\ge  h^{d-2}(a_{ij}r_i^2 - r_ir_j\sqrt{a_{ij}^2+b_{ij}^2h^2} + c_{ij}r_i^2h^2).
\end{aligned}
\end{equation*}
The inequality \eqref{aest} indicates that $a_{ij}$ is positive. Since $\sqrt{a_{ij}^2 + b_{ij}^2h^2}\le a_{ij} +\frac{b_{ij}^2h^2}{2a_{ij}}$,
\begin{equation}\label{MBP:uj}
\ell_i(U_j; \vec{v}_{ij})\ge a_{ij}r_i^2h^{d-2} - a_{ij}r_ir_jh^{d-2} + (c_{ij}-\frac{b_{ij}^2}{2a_{ij}})r_i^2h^d\ge (c_{ij}-\frac{b_{ij}^2}{2a_{ij}})r_i^2h^d,
\end{equation}
and the equation holds only if $r_i=r_j$ and $b_{ij}=0$. It follows that
\begin{equation}
R^1_i=-\sum_{j\neq i}\ell_i(U_j; \vec{v}_{ij})\le -\sum_{j\neq i}(c_{ij}-\frac{b_{ij}^2}{2a_{ij}})r_i^2h^2.
\end{equation}
When $ \mu_n\ge \frac{3C_2\|\bA_h^n\|_{0,\omega_i}^2}{C_1|\omega_i|}$, it follows \eqref{sumbsquare} and \eqref{sumcsquare} that 
$$
\sum_{j\neq i}\frac{b_{ij}^2}{2a_{ij}}\le \frac{C_2}{2C_1}h^{-d}\|\bA_h^n\|_{0,\omega_{i}}^2< \sum_{j\neq i}c_{ij},
$$
which implies that
\begin{equation}\label{Ri1}
R^1_i=-\sum_{j\neq i}\ell_i(U_j; \vec{v}_{ij}) <0.
\end{equation}
Since  $\phi_i\ge 0$ and $\sum_{j\neq i}\phi_j=1-\phi_i$  and $|\phi_i|_{0,1,\Omega}=\int_\Omega \phi_i\dx=\frac13|\omega_i|$,
\begin{equation*}
\begin{split}
R^2_i&\le \sum_{j\neq i}\int_\Omega |\bA_h^n|^2\phi_i\phi_j\,dx T_{ii}^{\rm re} - \int_\Omega |\bA_h^n|^2\phi_i^2\,dx T_{ii}^{\rm re} -   \mu_n|\phi_i|_{0,1,\Omega} T_{ii}^{\rm re}
\\
&= -2\int_\Omega |\bA_h^n|^2(\phi_i-\frac14)^2\dx T_{ii}^{\rm re}
-(\frac{1}{3}\mu_n|\omega_i| - \frac18\|\bA_h^n\|_{0,\omega_i}^2)\,dx T_{ii}^{\rm re}.
\end{split}
\end{equation*}
It follows from $T_{ii}^{\rm re}>0$ and $\mu_n\ge \frac{3\|\bA_h^n\|_{0,\omega_i}^2}{8|\omega_i|}$ that
$
R^2_i\le 0$.
A substitution of $R^2_i\le 0$ and \eqref{Ri1} into \eqref{Reterms} leads to
${\rm Re}(U_{i}^*\sum_{j=1}^{N}L_{{i}j}U_j)<0$, which verifies the assumption (a) in Lemma \ref{lm:complex}.
The assumption (b) in Lemma \ref{lm:complex} holds following the negative definite property \eqref{negativedef} of the matrix $L_{\mu_n, h}^n$.
For any $x_1, x_2\in \mathbb{C}$ with the magnitude not larger than 1, it is easy to verify that
$
|f_{\mu_n}(x_1)-f_{\mu_n}(x_2)|\le 2\mu_n |x_1-x_2|.
$
As proved in \cite{du2019maximum},
$$
|f_{\mu_n}(x_1)|=f_{\mu_n}(|x_1|)\le \mu_n\quad \mbox{if}\quad \mu_n\ge 2,
$$
which verifies the assumption (c) in Lemma \ref{lm:complex} with $\beta=1$ and completes the proof.
\end{proof}
\begin{remark}
Consider the stabilization parameter $\mu_n$ in  \eqref{mucondition}. 
The value of $\mu_n$ mainly depends on the value of $\displaystyle\max_{1\le i\le N}\frac{\|\bA_h^n\|_{0,\omega_i}^2}{|\omega_i|}$.
Note that $\sum_{i=1}^N\|\bA_h^n\|_{0,\omega_i}^2$ is bounded by a multiple of $\|\bA_h^n\|_{0,\Omega}^2$ from both above and below. This, together with the error estimate in Theorem \ref{th:err1} and the fact that $N=\mathcal{O}(h^{-d})=\mathcal{O}(|\omega_i|^{-1})$,  implies that there exists positive constants $c_1$ and $c_2$ such that
$$
c_1\|\bA(t_n)\|_{0,\Omega}^2\le \frac1N\sum_{i=1}^N\frac{\|\bA_h^n\|_{0,\omega_i}^2}{|\omega_i|}\le c_2\|\bA(t_n)\|_{0,\Omega}^2.
$$
The stabilization parameter $\mu_n$ depends on the maximum of $\frac{\|\bA_h^n\|_{0,\omega_i}^2}{|\omega_i|}$, where its average is bounded by $\|\bA(t_n)\|_{0,\Omega}^2$. Thus, the value of the parameter $\mu_n$ depends on the regularity of $\bA_h^n$, and usually will be bounded when the exact solution $\bA(t_n)$ is not too singular. Note that the approximation $\bA_h^n$ is already known when generating the stabilization parameter $\mu_n$ for the computation of $\Psi_h^n$. Thus, we can always find a stabilization parameter $\mu_n$ satisfying the condition \eqref{mucondition} to guarantee the discrete MBP even if the solution is not smooth.

\end{remark}

\subsection{Discrete energy stability}\label{sec:energy}

Define the discrete energy $G_h^n$ in an analogue form to \eqref{energy:def} by
\begin{equation*}
G_h^n=\frac12\|({i\over \kappa}\nabla +  \bA_h^n)\psi_h^n\|_0^2
+ \frac12\|\nabla\times \bA_h^n - \bH^n\|_0^2
+ \frac14 \||\Psi_h^n|^2-1\|_{\ell^2}^2,
\end{equation*}
and $\bM_h^n=\frac{1}{4\pi}(\nabla\times\bA_h^n-\bH^n)$.
\begin{theorem}\label{th:energy}
For any positive $\{\tau_n\}_{n=1}^{K_t}$, the solution $\{(\bA_h^n, \psi_h^n)\}_{n=0}^{K_t}$ generated by the discrete system \eqref{dis1}-\eqref{psieq} with  stabilization parameter $\mu_n$ satisfying \eqref{mucondition}  satisfies the energy inequality
\begin{equation*}
d_t^nG_h + \|d_t^n\bA_h\|_0^2
+ (\mu_n-1)\tau_n \|d_t^n\Psi_h\|_{\ell^2}^2
\le -4\pi(\bM_h^n, d_t^n\bH),\quad \forall 1\le n\le K_t.
\end{equation*}
Furthermore, if $\bH$ is independent of $t$,
we have
\begin{equation*}
G_h^n\le G_h^{n-1},\quad \forall 1\le n\le K_t,
\end{equation*}
i.e., the proposed scheme is unconditionally energy stable.
\end{theorem}
\begin{proof}
The difference between discrete energies at two consecutive time levels yields
\begin{equation*}
\begin{split}
d_t^nG_h=\frac12d_t^n\|({i\over \kappa}\nabla +  \bA_h )\psi_h \|_0^2
+ \frac12d_t^n\|4\pi\bM_h\|_0^2
+ \frac14d_t^n \||\Psi_h|^2-1\|_{\ell^2}^2.
\end{split}
\end{equation*}
It follows from \eqref{Bdef} that
$$d_t^n\|({i\over \kappa}\nabla +  \bA_h )\psi_h \|_0^2=\frac{1}{\tau_n}(B(\bA_h^n,\psi_h^n,\psi_h^n) - B(\bA_h^{n-1},\psi_h^{n-1},\psi_h^{n-1})),$$ and therefore,
\begin{equation} \label{energy:psi}
\begin{aligned}
\frac12 d_t^n\|({i\over \kappa}\nabla +  \bA_h )\psi_h \|_0^2
= &\frac{1}{2\tau_n}\left(B(\bA_h^n,\psi_h^n,\psi_h^n) - B(\bA_h^{n},\psi_h^{n-1},\psi_h^{n-1})\right)
\\
&
+ \frac12(d_t^n|\bA_h|^2, |\psi_h^{n-1}|^2)
+ (g(\psi_h^{n-1}), d_t^n\bA_h).
\end{aligned}
\end{equation}
Note that $$\frac12 d_t^n|u|^2 ={\rm Re}(u^n,d_t^nu)-\frac{\tau_n}{2}|d_t^nu|^2.$$ Thus,
\begin{equation}\label{iden1}
\begin{aligned}
(\bM^n,d_t^n\bM_h)
=&\frac12d_t^n\|\bM_h\|^2  + \frac{\tau_n}{2}\|d_t^n\bM_h\|^2.
\end{aligned}
\end{equation}
Let $\bB_h=d_t^n \bA_h$ in the scheme \eqref{dis1}. It holds that
\begin{equation*}
\|d_t^n \bA_h\|_0^2 + (4\pi\bM_h^n, 4\pi d_t^n\bM_h + d_t^n\bH) + (|\psi_h^{n-1}|^2\bA_h^n,d_t^n\bA_h)= - (g(\psi_h^{n-1}), d_t^n\bA_h).
\end{equation*}
A summation of \eqref{energy:psi}, \eqref{iden1} and the equation above yields
\begin{equation*}
\begin{split}
&\frac12 d_t^n\|({i\over \kappa}\nabla +  \bA_h )\psi_h \|_0^2
+ \frac12 d_t^n\|4\pi\bM_h\|_0^2
+ \|d_t^n\bA_h\|_0^2
+ \frac{\tau_n}{2}\|4\pi d_t^n\bM_h\|_0^2
\\
=&\frac{1}{2\tau_n}\left(B(\bA_h^n,\psi_h^n,\psi_h^n) - B(\bA_h^{n},\psi_h^{n-1},\psi_h^{n-1})\right)
- (4\pi\bM_h^n, d_t^n\bH)
\\
&
-(|\psi_h^{n-1}|^2\bA_h^n, d_t^n\bA_h)+ \frac12 (d_t^n|\bA_h|^2, |\psi_h^{n-1}|^2).
\end{split}
\end{equation*}
Note that  $$(|\psi_h^{n-1}|^2\bA_h^n, d_t^n\bA_h) - \frac12 (d_t^n|\bA_h|^2, |\psi_h^{n-1}|^2) = \frac{\tau_n}{2}\||\psi_h^{n-1}|d_t^n\bA_h\|_0^2\ge 0$$ and
\begin{equation*}
B(\bA_h^n;\phi_h,\phi_h)=-(I_h\phi_h)^H\hat L^n(I_h\phi_h)= (L_{0,h}^n(I_h\phi_h), I_h\phi_h)_{\ell^2}, \quad \forall \phi_h\in V_h.
\end{equation*}
It follows that
\begin{equation}\label{energy1}
\begin{split}
&\frac12 d_t^n\|({i\over \kappa}\nabla +  \bA_h )\psi_h \|_0^2
+ \frac12 d_t^n\|4\pi\bM_h\|_0^2
+ \|d_t^n\bA_h\|_0^2
+ \frac{\tau_n}{2}\|4\pi d_t^n\bM_h\|_0^2
\\
\le &-\frac{1}{2\tau_n}\left((L_{0,h}^n\Psi_h^n,\Psi_h^n)_{\ell^2} -(L_{0,h}^n\Psi_h^{n-1},\Psi_h^{n-1})_{\ell^2}\right)
- (4\pi\bM_h^n, d_t^n\bH).
\end{split}
\end{equation}
By \eqref{negativedef},
\begin{equation}\label{psi:ineq}
\begin{aligned}
&-\left((L_{0,h}^n\Psi_h^n,\Psi_h^n)_{\ell^2} -(L_{0,h}^n\Psi_h^{n-1},\Psi_h^{n-1})_{\ell^2}\right)
\\
=& -2\tau_n Re(L_{0,h}^n\Psi_h^{n} ,d_t^n\Psi_h)_{\ell^2} + \tau_n^2 Re(L_{0,h}^nd_t^n\Psi_h,d_t^n\Psi_h)_{\ell^2}\\
\le & -2\tau_n Re(L_{0,h}^n\Psi_h^{n} ,d_t^n\Psi_h)_{\ell^2}.
\end{aligned}
\end{equation}
Suppose $a$ and $b$ are complex numbers and $|a|\le 1$, $|b|\le 1$. It holds that
\begin{align*}
\frac14 ((a^2-1)^2 - (b^2-1)^2)
\le &(b^2-1) Re(b^*(a-b)) + (a-b)^*(a-b),
\end{align*}
which implies that for any $\mu_n\ge 1$,
\begin{equation}\label{psidiff}
\frac14 d_t^n\||\Psi_h|^2-1\|_{\ell^2}^2
+ (\mu_n-1)\tau_n\|d_t^n\Psi_h\|_{\ell^2}^2
\le Re(\mu_n \Psi_h^n - f_{\mu_n}(\Psi_h^{n-1}),d_t^n\Psi_h)_{\ell^2}.
\end{equation}
Substituting \eqref{psi:ineq} and \eqref{psidiff} into \eqref{energy1} yields
\begin{equation}\label{ETD1psienergy}
\begin{split}
&d_t^nG_h + \|d_t^n\bA_h\|_0^2
+ (\mu_n-1)\tau_n \|d_t^n\Psi_h\|_{\ell^2}^2
+ \frac{\tau_n}{2}\|d_t^n(\nabla\times \bA_h - \bH)\|_0^2
\\
\le&- Re(f_{\mu_n}(\Psi_h^{n-1}) + L_{\mu_n,h}^n\Psi_h^{n},d_t^n\Psi_h)_{\ell^2}
- (4\pi\bM_h^n, d_t^n\bH).
\end{split}
\end{equation}
The ETD1 scheme in \eqref{ETDsol} indicates that
\begin{align*}
f_{\mu_n}(\Psi_h^{n-1})= &-(1-e^{L_{\mu_n,h}^n\tau_n})^{-1} L_{\mu_n,h}^n (\Psi_h^n - e^{L_{\mu_n,h}^n\tau_n}\Psi_h^{n-1})
\\
= &-(1-e^{L_{\mu_n,h}^n\tau_n})^{-1} L_{\mu_n,h}^n (\Psi_h^n -\Psi_h^{n-1} + (I- e^{L_{\mu_n,h}^n\tau_n})\Psi_h^{n-1})
\\
=&-\tau_n(1-e^{L_{\mu_n,h}^n\tau_n})^{-1} L_{\mu_n,h}^n d_t^n\Psi_h - L_{\mu_n,h}^n\Psi_h^{n-1}.
\end{align*}
Define $g(x)=-x+x/(1-e^{x})$ and the operator $\Delta_1= g_1(L_{\mu_n, h}^n\tau_n)$.  It follows that
\begin{equation*}
\begin{split}
-L_{\mu_n,h}^n\Psi_h^n - f_{\mu_n}(\Psi_h^{n-1})
=\Delta_1(d_t^n\Psi_h).
\end{split}
\end{equation*}
Since $g(x)<0$ for all $x<0$ and $L_{\mu_n,h}^n$ is self-adjoint and negative definite, the operator $\Delta_1$ is also negative definite. Thus,
$$
- Re(f_{\mu_n}(\Psi_h^{n-1}) + L_{\mu_n,h}^n\Psi_h^{n}, d_t^n\Psi_h)\le 0,
$$
which combined with \eqref{ETD1psienergy} gives
\begin{equation*}
d_t^nG_h + \|d_t^n\bA_h\|_0^2
+ (\mu_n-1)\tau_n \|d_t^n\Psi_h\|_{\ell^2}^2
+ \frac{\tau_n}{2}\|4\pi d_t^n\bM_h\|_0^2
\le - (4\pi\bM_h^n, d_t^n\bH).
\end{equation*}
If $\bH$ is stationary, the right-hand side of the above inequality equals zero, which indicates
$
G_h^n \le G_h^{n-1}
$
and completes the proof.
\end{proof}

\section{Error estimate}\label{sec:estimate}
In this section, we analyze the convergence of the numerical solutions by the proposed scheme \eqref{dis1}-\eqref{psieq} under the regularity assumption below.

\begin{assumption}\label{ass:regularity}
Assume that $\Omega$ is a convex polygon (or polyhedron). The solution of the initial boundary value problem \eqref{model0temporal} with \eqref{model0bc} satisfies the regularity conditions
\begin{equation*}
\begin{split}
&\psi,\ \partial_t\psi\in L^\infty(0, T; H^2(\Omega, \mathbb{C})),\quad \bA,\ \partial_t\bA\in L^\infty(0, T; V_A), \\
&\partial_{tt}\bA\in  L^\infty(0, T; H^1(\Omega, \mathbb{R}^d)).
\end{split}
\end{equation*}
where $V_A=\{\bB\in H^1(\Omega, \mathbb{R}^d): \nabla\times \bB\in H^1(\Omega, \mathbb{R}^d)\}$.
\end{assumption}

To begin with, we explore the relation between the errors $e_{\bA}^n$ and $E_{\psi}^n$ at two consecutive time levels by use of the error equations, where
\begin{equation*}
	e_{\bA}^j=\bA_h^j - R_{h}\bA^j,
	\quad E_{\psi}^j=\Psi_h^j- I_h\psi^j,
	\quad e_{\psi}^j=\Pi_hE_{\psi}^j.
\end{equation*}
By the estimate \eqref{ritz:err} and the interpolation error of the linear element
\begin{equation}\label{err:linear}
\|\bA^n - R_h\bA^n\|_0 + h\|\nabla\times (\bA^n - R_h\bA^n)\|_0 + \|\psi - \Pi_L\psi\|_0+h\|\nabla(\psi-\Pi_L\psi)\|_0\lesssim h^2.
\end{equation}

\begin{lemma}
Assume that  Assumption~\ref{ass:mesh} and \ref{ass:regularity} hold. Let $\bA_h^0=R_{h}\bA^0$ and $\Psi_h^0=I_h\psi^0$ with $\|\psi^0\|_\infty\le 1$. The approximation solution $\{(\bA_h^n, \Psi_h^n)\}_{n=1}^{K_t}$ is generated by the numerical scheme \eqref{dis1}-\eqref{psieq} with  stabilization parameter $\mu_n$ satisfying \eqref{mucondition}  and uniform time step $\tau_n=\tau$. For any $1\le n\le K_t$,
\begin{equation}\label{eq:lme}
\begin{split}
\sigma \|e_{\bA}^n\|^2 + 2\tau \|\nabla \times e_{\bA}^n\|_0^2
\le& \sigma (1+C\tau)\|e_{\bA}^{n-1}\|^2
+ \tau\|({\mbox{i}\over \kappa}\nabla + \bA_h^{n-1}) e_{\psi}^{n-1}\|_0^2
\\
&+ C\tau \left(\|E_{\psi}^{n-1}\|_{\ell^2}^2+  h^2+ \tau^2\right),
\end{split}
\end{equation}
\begin{equation}\label{eq:lmecurl}
\begin{split}
\|\nabla \times e_{\bA}^n\|_0^2
\le&  \|\nabla \times e_{\bA}^{n-1}\|_0^2
+ C\tau(\|({\mbox{i}\over \kappa}\nabla + \bA_h^{n-1}) e_{\psi}^{n-1}\|_0^2 + \|e_{\bA}^{n-1}\|_0^2 + \|e_{\bA}^{n}\|_0^2
\\
& + \|E_{\psi}^{n-1}\|_{\ell^2}^2+ h^2+ \tau^2).
\end{split}
\end{equation}
\end{lemma}
\begin{proof}

By the definition of the Ritz projection $R_{h}$ in  \eqref{ritz} and \eqref{dis1},
\begin{equation}\label{Aidentity}
\begin{split}
&\sigma (d_t^n e_{\bA}, \bB_h)
+ (\nabla \times e_{\bA}^n, \nabla \times \bB_h)
\\
&+  (Re[(\psi_h^{n-1})^*({\mbox{i}\over \kappa}\nabla + \bA_h^n) \psi_h^{n-1} - (\psi^{n})^*({\mbox{i}\over \kappa}\nabla + \bA^n) \psi^{n}], \bB_h)
\\
=& \sigma (\partial_t^n \bA - d_t^n \bA, \bB_h)
+ \sigma (  (I -R_{h})d_t^n\bA, \bB_h)
- ((I- R_{h})\bA^n, \bB_h).
\end{split}
\end{equation}
Since $$\partial_t^n \bA - d_t^n \bA=\frac{1}{\tau}\int_{t_{n-1}}^{t_n} \partial_t^n\bA -\partial_t\bA(s)\ds,$$
\begin{equation}\label{Aest1term}
|(\partial_t^n \bA - d_t^n \bA, \bB_h)|\lesssim \tau\|\bB_h\|_0.
\end{equation}
By the estimate \eqref{ritz:err},
\begin{equation}\label{Aest2term}
|\sigma((I -R_{h})d_t^n\bA, \bB_h)| + |((I- R_{h})\bA^n, \bB_h)|\lesssim h\|\bB_h\|_0.
\end{equation}
Note that
\begin{equation*}\label{deco}
\begin{aligned}
&(\psi_h^{n-1})^*({\mbox{i}\over \kappa}\nabla + \bA_h^n) \psi_h^{n-1} - (\psi^{n})^*({\mbox{i}\over \kappa}\nabla + \bA^n) \psi^{n}
\\
=&(e_{\psi}^{n-1})^*({\mbox{i}\over \kappa}\nabla + \bA^n) \Pi_L\psi^{n-1}
+ (\psi_h^{n-1})^*({\mbox{i}\over \kappa}\nabla + \bA^n) e_{\psi}^{n-1}
+ (\psi_h^{n-1})^*(\bA_h^n - \bA^n) \psi_h^{n-1}
\\
&
+ (\Pi_L\psi^{n-1})^*({\mbox{i}\over \kappa}\nabla + \bA^n) \Pi_L\psi^{n-1}
- (\psi^{n})^*({\mbox{i}\over \kappa}\nabla + \bA^n) \psi^{n},
\end{aligned}
\end{equation*}
where Assumption \ref{ass:regularity} and the error estimates in \eqref{err:linear} and \eqref{ritz:err} imply that
\begin{equation*}
\begin{aligned}
&|(e_\psi^{n-1})^*({\mbox{i}\over \kappa}\nabla + \bA^n) \Pi_L\psi^{n-1},\bB_h)|
\lesssim \|e_{\psi}^{n-1}\|_0\|\bB_h\|_0,
\\
&|((\psi_h^{n-1})^*({\mbox{i}\over \kappa}\nabla + \bA^n) e_\psi^{n-1},\bB_h)|
\le \|({\mbox{i}\over \kappa}\nabla + \bA^n) e_{\psi}^{n-1}\|_0\|\bB_h\|_0\|\psi_h^{n-1}\|_\infty,
\\
&|((\psi_h^{n-1})^*(\bA_h^n - \bA^n) \psi_h^{n-1},\bB_h)|
\le (\|e_{\bA}^n\|_0 + Ch)\|\bB_h\|_0\|\psi_h^{n-1}\|_\infty^2,
\\
&|((\Pi_L\psi^{n-1})^*({\mbox{i}\over \kappa}\nabla + \bA^n) \Pi_L\psi^{n-1}
- (\psi^{n})^*({\mbox{i}\over \kappa}\nabla + \bA^n) \psi^{n}, \bB_h)|
\lesssim (\tau + h)\|\bB_h\|_0.
\end{aligned}
\end{equation*}
By Theorem \ref{th:MBP1}, $\|\psi_h^{n-1}\|_\infty\le 1$. It follows that
\begin{equation}\label{realest}
\begin{aligned}
&|(Re[(\psi_h^{n-1})^*({\mbox{i}\over \kappa}\nabla + \bA_h^n) \psi_h^{n-1} - (\psi^{n})^*({\mbox{i}\over \kappa}\nabla + \bA^n) \psi^{n}], \bB_h)|
\\
\le&
\left(\|({\mbox{i}\over \kappa}\nabla + \bA^n) e_{\psi}^{n-1}\|_0
+ \|e_{\bA}^n\|_0
+ C\|e_{\psi}^{n-1}\|_0+ C\tau + Ch\right)\|\bB_h\|_0.
\end{aligned}
\end{equation}
It follows from $\| e_{\psi}^{n-1}\|_\infty \le \|I_h\psi^{n-1}\|_\infty + \|\psi_h^{n-1}\|_\infty\le 2$ and \eqref{ritz:err} that
\begin{equation}\label{extra_1}
\begin{aligned}
\|({\mbox{i}\over \kappa}\nabla + \bA^n) e_{\psi}^{n-1}\|_0
\le& \|({\mbox{i}\over \kappa}\nabla + \bA_h^{n-1}) e_{\psi}^{n-1}\|_0
+ \|(\bA^n - \bA_h^{n-1}) e_{\psi}^{n-1}\|_0
\\
\le &\|({\mbox{i}\over \kappa}\nabla + \bA_h^{n-1}) e_{\psi}^{n-1}\|_0
+ 2\|e_{\bA}^{n-1}\|_0 + C(\tau + h).
\end{aligned}
\end{equation}
Let $\bB_h=e_{\bA}^n$ in \eqref{Aidentity}. By Young's inequality, a combination of \eqref{Aidentity}, \eqref{Aest1term}, \eqref{Aest2term}, \eqref{realest}  and \eqref{extra_1} leads to
\begin{equation*}
\begin{split}
&\sigma \|e_{\bA}^n\|^2 + 2\tau \|\nabla \times e_{\bA}^n\|_0^2
\\
\le& \sigma (1+C\tau)\|e_{\bA}^{n-1}\|^2
+ \tau \|({\mbox{i}\over \kappa}\nabla + \bA_h^{n-1}) e_{\psi}^{n-1}\|_0^2
+ C\tau \left( \|e_{\psi}^{n-1}\|_0^2+ \tau^2 + h^2\right).
\end{split}
\end{equation*}
Let $\bB_h=d_t^ne_{\bA}$ in \eqref{Aidentity}. A similar analysis yields
\begin{equation*}
\begin{split}
\|\nabla \times e_{\bA}^n\|_0^2
\le&  \|\nabla \times e_{\bA}^{n-1}\|_0^2
+ C\tau(\|({\mbox{i}\over \kappa}\nabla + \bA_h^{n-1}) e_{\psi}^{n-1}\|_0^2
+ \|e_{\bA}^{n-1}\|_0^2 + \|e_{\bA}^{n}\|_0^2
\\
& + \|e_{\psi}^{n-1}\|_0^2+ \tau^2 + h^2),
\end{split}
\end{equation*}
which completes the proof.
\end{proof}

Given any $\mu\ge 0$ and $\bB\in H^1(\Omega)$, denote the linear  operator $L_{\mu}[\bB]\psi=-(\frac{\mbox{i}}{\kappa}\nabla + \bB)^2\psi - \mu\psi$. The matrix $L_{\mu_n,h}^n$ in \eqref{psieq} relates to a spatial discretization of the operator $L_{\mu_n}[\bA^n]$.
Let $S_\psi(t)=U_h(t)- \Psi(t)$ with $U_h$ defined in \eqref{psieq} and $\Psi(t)= I_h\psi(\cdot, t)$. A subtraction of \eqref{psieq} from \eqref{model0temporal} reads
\begin{equation}\label{Sheq2}
\left\{
\begin{aligned}
&\frac{d}{dt}S_\psi = L_{\mu_n,h}^n S_\psi + \delta_n^1 + \delta_n^2 + \delta_n^3 + f_{\mu_n}(\Psi_h^{n-1}) - f_{\mu_n} (I_h\psi^{n-1}), t\in [t_{n-1}, t_n],\\
&S_\psi(t_{n-1}) = E_\psi^{n-1},
\end{aligned}
\right.
\end{equation}
where
\begin{equation}\label{deltadef}
\begin{aligned}
\delta_n^1&=L_{\mu_n,h}^n I_h\psi - I_hL_{\mu_n}[\bA^n]\psi, &
\delta_n^2&=I_h(L_{\mu_n}[\bA^n] - L_{\mu_n}[\bA])\psi,
\\
\delta_n^3&=f_{\mu_n}(I_h\psi^{n-1}) - f_{\mu_n} (I_h\psi).&
\end{aligned}
\end{equation}
The first term $\delta_n^1$ represents the consistency error of the numerical scheme \eqref{psieq} and the other two terms relate to the error in time discretization.

\begin{lemma}\label{lm:delta}
Under Assumption \ref{ass:regularity}, it holds for any $W_h\in \mathbb{C}^N$ that
\begin{align*}
|(\delta_n^1, W_h)_{\ell^2} |\lesssim (\|e_{\bA}^{n}\|_0 + h)(\|(\frac{i}{\kappa}\nabla + \bA_h^{n})\Pi_hW_h\|_0 + \|\Pi_hW_h\|_0).
\end{align*}
\end{lemma}
\begin{proof}
Let  $w_h=\Pi_hW_h$. It follows from \eqref{L2def} that
\begin{equation}\label{psi:constency0}
|(I_hL_{\mu_n}[\bA^n]\psi, W_h)_{\ell^2}
- (L_{\mu_n}[\bA^n]\psi, w_h)|\lesssim h\|(\frac{\mbox{i}}{\kappa}\nabla + \bA^n)^2\psi\|_1\|w_h\|_0.
\end{equation}
By the definition of $L_{\mu_n,h}[\bA_h^{n}]$ in \eqref{ETD:L},
\begin{equation*}
(L_{\mu_n,h}^n I_h\psi, W_h)_{\ell^2}= -((\frac{\mbox{i}}{\kappa}\nabla + \bA_h^{n})\Pi_L\psi, (\frac{\mbox{i}}{\kappa}\nabla + \bA_h^{n})w_h)
- \mu_n(\Pi_L\psi, w_h).
\end{equation*}
It follows from the above equation and the integration by parts that
\begin{equation}\label{psi:constency}
\begin{aligned}
&(L_{\mu_n,h}^n I_h\psi, W_h)_{\ell^2}
- (L_{\mu_n}[\bA^{n}]\psi, w_h)
=&\sum_{i=1}^5I_i,
\end{aligned}
\end{equation}
where $I_1= (\frac{\mbox{i}}{\kappa}\nabla (\psi-\Pi_L\psi), t_h)$, $I_2= (\bA^n(\psi - \Pi_L\psi), t_h)$, $I_3= ((\bA^n - \bAt_h^n)\Pi_L\psi, t_h)$, $I_4= ((\frac{\mbox{i}}{\kappa}\nabla + \bA^n)\psi, (\bA^n - \bA_h^{n})w_h)$ and $I_5=\mu_n((I-\Pi_L)\psi,w_h)$ with $t_h=(\frac{\mbox{i}}{\kappa}\nabla + \bA_h^{n})w_h$.
It follows from the estimate \eqref{err:linear} that
\begin{equation}\label{err:con1}
\begin{aligned}
|I_1| + |I_2|  + |I_3| + |I_4| + |I_5|
\lesssim (\|\et_{\bA}^{n}\|_0 + h)\|t_h\|_0 + (\|\et_{\bA}^{n}\|_0 + h^2)\|w_h\|_0.
\end{aligned}
\end{equation}
This, together with \eqref{psi:constency0} and \eqref{psi:constency}
, leads to
\begin{equation*}
\begin{aligned}
\left| (\delta_n^1, W_h)_{\ell^2}\right|
\lesssim &(\|e_{\bA}^{n}\|_0 + h)\|(\frac{i}{\kappa}\nabla + \bA_h^{n})w_h\|_0 + (\|e_{\bA}^{n}\|_0 + h)\|w_h\|_0,
\end{aligned}
\end{equation*}
which completes the proof.
\end{proof}

\begin{lemma}\label{lm:ETD1err}
Assume that  Assumption~\ref{ass:mesh} and \ref{ass:regularity} hold. Let $\bA_h^0=R_{h}\bA^0$ and $\Psi_h^0=I_h\psi^0$ with $\|\psi^0\|_\infty\le 1$. $\{(\bA_h^n, \Psi_h^n)\}_{n=1}^{K_t}$ is generated by the discrete system \eqref{dis1}-\eqref{psieq} with  the stabilizing parameter  stabilization parameter $\mu_n$ satisfying \eqref{mucondition} and time step $\tau_n=\tau$. For any $1\le n\le K_t$,
\begin{equation}\label{eq:lms}
\|E_\psi^{n}\|_{\ell^2}^2 + \tau \|(\frac{i}{\kappa}\nabla + \bA_h^{n})e_\psi^{n}\|_0^2
\le (1+C\tau)\|E_\psi^{n-1}\|_{\ell^2}^2 + C\tau(\|e_{\bA}^{n}\|_0^2 + \tau^2 + h^2).
\end{equation}
\end{lemma}
\begin{proof}
It follows from \eqref{Sheq2} that
\begin{equation*}
E_\psi^{n}=e^{\tau L_{\mu_n,h}^n}E_\psi^{n-1} + \int_0^\tau e^{(\tau-s)L_{\mu_n,h}^n}(\delta_n^1 + \delta_n^2 + \delta_n^3 + f_{\mu_n}(\Psi_h^{n-1}) - f_{\mu_n} (I_h\psi^{n-1}))\ds.
\end{equation*}
Acting $I-\tau L_{\mu_n,h}^n$ on both sides of the equation above and taking $\ell^2$ inner product with $E_\psi^{n}$ yield
\begin{equation}\label{S0}
\begin{aligned}
&\|E_\psi^{n}\|_{\ell^2}^2 + \tau \|(\frac{i}{\kappa}\nabla + \bA_h^{n})e_\psi^{n}\|_0^2 + \mu \tau\|E_\psi^{n}\|_{\ell^2}^2
\\
=&(q_1(\tau L_{\mu_n,h}^n)E_\psi^{n-1} + \tau q_2(\tau L_{\mu_n,h}^n)(f_{\mu_n}(\Psi_h^{n-1}) - f_{\mu_n} (I_h\psi^{n-1})),E_\psi^{n})_{\ell^2}
\\
&+ \int_0^\tau ((I-\tau L_{\mu_n,h}^n)e^{(\tau-s)L_{\mu_n,h}^n}\delta_n^1,E_\psi^{n})_{\ell^2}\ds
\\
&+ \int_0^\tau ((I-\tau L_{\mu_n,h}^n)e^{(\tau-s)L_{\mu_n,h}^n}(\delta_n^2 + \delta_n^3),E_\psi^{n})_{\ell^2}\ds,
\end{aligned}
\end{equation}
where $q_1(x)=(1-x)e^x$, $q_2(x)={(1-x)(e^x-1)}/{x}$. Note that for any $x<0$,
\begin{align*}
0<q_1(x)<1<q_2(x)<2.
\end{align*}
Since $L_{\mu_n,h}^n$ is negative definite,
\begin{equation}\label{S1}
\begin{aligned}
&|(q_1(\tau L_{\mu_n,h}^n)E_\psi^{n-1},E_\psi^{n})_{\ell^2}|\le  \|E_\psi^{n-1}\|_{\ell^2}\|E_\psi^{n}\|_{\ell^2},
\\
&|\tau (q_2(\tau L_{\mu_n,h}^n)(f_{\mu_n}(U_h^{n-1}) - f_{\mu_n} (I_h\psi^{n-1})),E_\psi^{n})_{\ell^2}|\le  C \tau\|E_\psi^{n-1}\|_{\ell^2}\|E_\psi^{n}\|_{\ell^2}.
\end{aligned}
\end{equation}
It follows from Lemma \ref{lm:delta} that
\begin{equation*}
|((I-\tau L_{\mu_n,h}^n)e^{(\tau-s)L_{\mu_n,h}^n}\delta_n^1,E_\psi^{n})_{\ell^2}|
\lesssim (\|e_{\bA}^{n}\|_0 + h)(\|(\frac{i}{\kappa}\nabla + \bA_h^{n})t_h^n\|_0
+ \|t_h^n\|_0),
\end{equation*}
where $$t_h^n=(I-\tau (L_{\mu_n,h}^n)^T)e^{(\tau-s)(L_{\mu_n,h}^n)^T}E_\psi^n.$$ Since $0<q_1(x)<1$,
\begin{equation}\label{S2}
\begin{split}
&\left|\int_0^\tau ((I-\tau L_{\mu,h}^n)e^{(\tau-s)L_{\mu_n,h}^n}\delta_n^1,E_\psi^{n})_{\ell^2}\ds\right|\\
\lesssim &\tau(\|e_{\bA}^{n}\|_0 + h)(\|(\frac{\mbox{i}}{\kappa}\nabla + \bA_h^{n})e_\psi^n\|_0 + \|E_\psi^n\|_{\ell^2}).
\end{split}
\end{equation}
Note that $\|\delta_n^2\|_0 + \|\delta_n^3\|_0\lesssim \tau$. Thus,
\begin{equation}\label{S3}
\left|\int_0^\tau ((I-\tau L_{\mu_n,h}^n)e^{(\tau-s)L_{\mu_n,h}^n}(\delta_n^2 + \delta_n^3),E_\psi^{n})_{\ell^2}\ds\right|\lesssim \tau^2\|E_\psi^n\|_{\ell^2}.
\end{equation}
A substitution of \eqref{S1}, \eqref{S2} and \eqref{S3} into \eqref{S0} gives
\begin{equation*}
\begin{split}
&\|E_\psi^{n}\|_{\ell^2}^2 + \tau \|(\frac{\mbox{i}}{\kappa}\nabla + \bA_h^{n})e_\psi^{n}\|_0^2 + \frac{\mu_n\tau}{2} \|E_\psi^{n-1}\|_{\ell^2}^2
\\
\le &(1+ C \tau)\|E_\psi^{n-1}\|_{\ell^2}\|E_\psi^{n}\|_{\ell^2}
+ C\tau(\|e_{\bA}^{n}\|_0 + h)\|(\frac{\mbox{i}}{\kappa}\nabla + \bA_h^{n})e_\psi^n\|_0
\\
&+ C\tau\|E_\psi^n\|_0(\|e_{\bA}^{n}\|_{\ell^2} + \|E_\psi^n\|_{\ell^2} + \tau + h).
\end{split}
\end{equation*}
By the Young's inequality,
\begin{equation*}
\|E_\psi^{n}\|_{\ell^2}^2  + \tau \|(\frac{\mbox{i}}{\kappa}\nabla + \bA_h^{n})e_\psi^{n}\|_0^2
\le (1+ C \tau)\|E_\psi^{n-1}\|_{\ell^2}^2 + C\tau(\|e_{\bA}^{n}\|_0^2 + \tau^2 + h^2),
\end{equation*}
which completes the proof.
\end{proof}

The following theorem presents the main result of the error estimate of the proposed numerical scheme \eqref{dis1}-\eqref{psieq}.

\begin{theorem}\label{th:err1}
Assume that  Assumption~\ref{ass:mesh} and \ref{ass:regularity} hold. Let $\bA_h^0=R_{h}\bA^0$ and $\Psi_h^0=I_h\psi^0$ with $\|\psi_0\|_\infty\le 1$. $\{(\bA_h^n, \Psi_h^n)\}_{n=1}^{K_t}$ is generated by the discrete system \eqref{dis1}-\eqref{psieq} with the stabilizing parameter  stabilization parameter $\mu_n$ satisfying \eqref{mucondition} and time step $\tau_n=\tau$. For any $1\le n\le K_t$,
\begin{equation*}
\|\bA_h^n-\bA^n\|_0 + \|\nabla\times (\bA_h^n-\bA^n)\|_0 + \|\psi_h^n-\psi^n\|_0\lesssim \tau + h.
\end{equation*}
\end{theorem}
\begin{proof}
Denote
$$
T^n=\|E_\psi^n \|_{\ell^2}^2
+\sigma \|e_{\bA}^n\|_0^2
+ \tau \|(\frac{\mbox{i}}{\kappa}\nabla + \bA_h^{n})e_\psi^{n}\|_0^2
+ 2\tau \|\nabla \times e_{\bA}^n\|_0^2.
$$
By the estimates \eqref{eq:lme} and \eqref{eq:lms},
\begin{equation*}
T^n \le (1+C\tau)T^{n-1}
+ C\tau (h^{2} + \tau^2),
\end{equation*}
which implies that
\begin{equation*}
T^n \le (1+C\tau)^n T^0 + C\tau (\tau^2 + h^2)\sum_{i=1}^n(1+C\tau)^i.
\end{equation*}
Note that there exists constant~$C_0$ such that
$$
|(1+C\tau)^n| + |\tau\sum_{i=1}^n(1+C\tau)^i|\le C_0.
$$
This, together with the fact that $|T^0|\lesssim h^2$, leads to
\begin{equation}\label{L2error}
\|e_{\bA}^n\|_0 + \|E_{\psi}^n \|_{\ell^2}\lesssim \tau + h.
\end{equation}
As a consequence, the estimate \eqref{eq:lms} reads
\begin{equation*}
\|E_\psi^n\|_{\ell^2}^2 + \tau\|({\mbox{i}\over \kappa}\nabla + \bA_h^{n}) e_{\psi}^{n}\|_0^2 \le (1+C\tau )\|E_\psi^{n-1}\|_{\ell^2}^2 + C\tau(\tau^2 + h^2),
\end{equation*}
which leads to
\begin{equation*}
\begin{split}
&\|E_\psi^n\|_{\ell^2}^2 + \tau\sum_{j=1}^n\|({\mbox{i}\over \kappa}\nabla + \bA_h^{j}) e_{\psi}^{j}\|_0^2\\
\le &(1+C\tau )^n\|E_\psi^{0}\|_{\ell^2}^2 + C\tau(\tau^2 + h^2)\sum_{j=0}^{n-1}(1+C\tau)^j
\le C(\tau^2 + h^2).
\end{split}
\end{equation*}
Substituting this into the estimate \eqref{eq:lmecurl} yields
\begin{equation*}
\begin{split}
\|\nabla \times e_{\bA}^n\|_0^2
\le&  \|\nabla \times e_{\bA}^{0}\|_0^2
+ C\tau\sum_{j=0}^{n-1}\|({\mbox{i}\over \kappa}\nabla + \bA_h^{j}) e_{\psi}^{j}\|_0^2
+ C(\tau^2 + h^2)
\le C(\tau^2 + h^2).
\end{split}
\end{equation*}
A combination of the estimate above and \eqref{L2error} gives
\begin{equation*}
\|e_{\bA}^n\|_0 + \|E_{\psi}^n \|_{\ell^2} + \|\nabla \times e_{\bA}^n\|_0\lesssim \tau + h.
\end{equation*}
This, together with the estimate \eqref{err:linear}, completes the proof.
\end{proof}

\begin{remark}\label{remark:ppr}
In the decoupled numerical scheme \eqref{dis1}-\eqref{psieq}, the first order convergence rate of  $\|\bA_h^n-\bA^n\|_0$ is one degree lower than that of the projection error $\|R_h\bA^n-\bA^n\|_0$ provided that $\bA\in H^2(\Omega,\mathbb{R}^d)$. The gap is caused by nonlinearity, that is the explicit gradient term $g(\psi_h^{n-1})$ in \eqref{dis1}. We can fix the gap by applying the gradient recovery technique in  \cite{zhang2005new} and replace  $g(\psi_h^{n-1})$ in \eqref{dis1} by the recovered gradient, that is to seek $(\hat{\bA_h^n}, \hat \psi_h^n)$ such that
\begin{equation*}
(d_t^n \hat{\bA_h},\bB_h) + D(\hat \psi^{n-1}_h; \hat{\bA_h^n}, \bB_h)=(\bH^n,\nabla\times \bB_h) - (g_M(\hat \psi_h^{n-1},\hat \psi_h^{n-1}), \bB_h),
\end{equation*}
for any $\bB_h\in Q_h$ and $\hat \Psi_h^n=\hat U_h(t_n)$ satisfying \eqref{psieq} with $\hat \Psi_h^0=I_h\psi^0$ and $\hat{ \bA_h^0}=R_{h}\bA^0$, where
$g_M(\psi_h,\psi_h)=\frac{\mbox{i}}{2\kappa} (\psi_h^*(K_h \psi_h) - \psi_h(K_h \psi_h^*))$
with recovered gradient $K_h\psi_h$.
\end{remark}

\begin{remark}
Note that the convergence analysis in Theorem \ref{th:err1} relies on the interpolation error of the solutions, thus the first order convergence rate does not hold theoretically for the numerical scheme when the domain is not convex. Nevertheless, the discrete MBP in Theorem \ref{th:MBP1} and the energy dissipation property in Theorem \ref{th:energy} still hold for non-convex superconductors.
\end{remark}

\section{Numerical Examples}\label{sec:numerical}

In this section, we present some numerical examples to verify the theoretical results and show the vortex motions of superconductors in an external magnetic field.

\subsection{Example 1: convergence test}
Consider the artificial example on $\Omega=(0,1)^2$ with $\kappa=1$
\begin{equation}\label{artificial}
\left\{
\begin{aligned}
\partial_{t} \psi&= -\left(\frac{\mbox{i}}{\kappa} \nabla + \boldsymbol{A}\right)^{2} \psi+\psi-|\psi|^{2} \psi + g& \text { in } \Omega,
\\
\partial_{t} \boldsymbol{A}&=\frac{1}{2 \mbox{i} \kappa }(\psi^* \nabla \psi-\psi \nabla \psi^*)-|\psi|^{2} \boldsymbol{A}-\nabla \times \nabla \times \boldsymbol{A} + f& \text { in } \Omega,
\end{aligned}
\right.
\end{equation}
and boundary and initial conditions \eqref{model0bc}. The functions $f$, $g$, $\psi^0$ and $\boldsymbol{A}^0$ are chosen corresponding to the exact solution
$
\psi = e^{-t}(\cos(2\pi x) + \mbox{i}\cos (\pi y))
$, $\boldsymbol{A}=  [e^tx^{1.001}(1-x)^{5/4}y,\ e^ty^{1.001}(1-y)^{1.001}x]^T  
$
with
$
\bH=\nabla\times \bA
$
We set the terminal time $T=1$ {  and the stabilization parameter $\mu_n=2$} in this example.
Table~\ref{tab:square1} records the $L^2$-norm errors of $\bA_h$, $\nabla\times \bA_h$, $\psi_h$ and $\nabla\psi_h$ on uniform triangulations with spatial mesh size $h$, which coincide with the convergence result in Theorem \ref{th:err1} and show the accuracy of the proposed numerical scheme when the solution is smooth enough.

\begin{table}[!htp]\scriptsize
	\centering
	\begin{tabular}{c|cc|cc|cc|cc}
		\hline
		$1/h$& $\|\bA-\bA_h\|_0$& rate&  $\|\nabla\times(\bA-\bA_h)\|_0$& rate& $\|\psi-\psi_h\|_0$&rate&  $\|\nabla(\psi-\psi_h)\|_0$&rate\\\hline
4& 1.81E+00 &       & 9.75E-01 &       & 8.18E-01 & & 1.28E+00 &  \\\hline
8& 1.31E+00 &0.46& 4.58E-01 & 1.09& 3.36E-01 & 1.28& 4.91E-01 & 1.38 \\\hline
16& 6.32E-01 &1.05& 2.29E-01 & 1.00& 2.23E-01 & 0.59& 2.21E-01 & 1.15\\\hline
32& 3.01E-01 &1.07& 1.14E-01 & 1.00& 1.26E-01 & 0.83& 1.07E-01 & 1.04 \\\hline
64& 1.48E-01 &1.02& 5.70E-02 & 1.00& 6.60E-02 & 0.93& 5.35E-02 & 1.01\\\hline
128& 7.39E-02 &1.00& 2.85E-02 & 1.00& 3.38E-02 & 0.97& 2.67E-02 & 1.00\\\hline
256& 3.70E-02 &1.00& 1.43E-02 & 1.00& 1.71E-02 & 0.98&1.34E-02& 1.00\\\hline
    \end{tabular}%
\caption{ Errors and convergence rates with time step $\tau=10^{-5}$.}
\label{tab:square1}
\end{table}%

\subsection{Example 2: L-shaped superconductor}
We use the proposed formulation to simulate the vortex dynamics in the superconductor $\Omega=(-0.5,0.5)^2\backslash[0, 0.5]\times[-0.5,0]$ with the Ginzburg--Landau parameter $\kappa=10$. The initial conditions and applied magnetic field are
$\psi^0 = 0.6+0.8{\rm i}$, $\boldsymbol{A}^0=(0,0)$ and $\bH=5.$
This example was tested before by different methods, see \cite{gao2017efficient,li2015new} for reference. We simulate the problem on a uniform triangulation with $M=16$ nodes per unit length on each side with stabilization parameter $\mu_n=2$ and time step $\tau=1/16$.  Fig. \ref{fig:Lenergy} plots the discrete energy of the proposed scheme and the maximum norm of the discrete order parameter, which verifies the theoretical results in Theorems~\ref{th:energy} and \ref{th:MBP1}.

\begin{figure}[!ht]
	\centering
\includegraphics[width = 1.9in,height=1.5in]{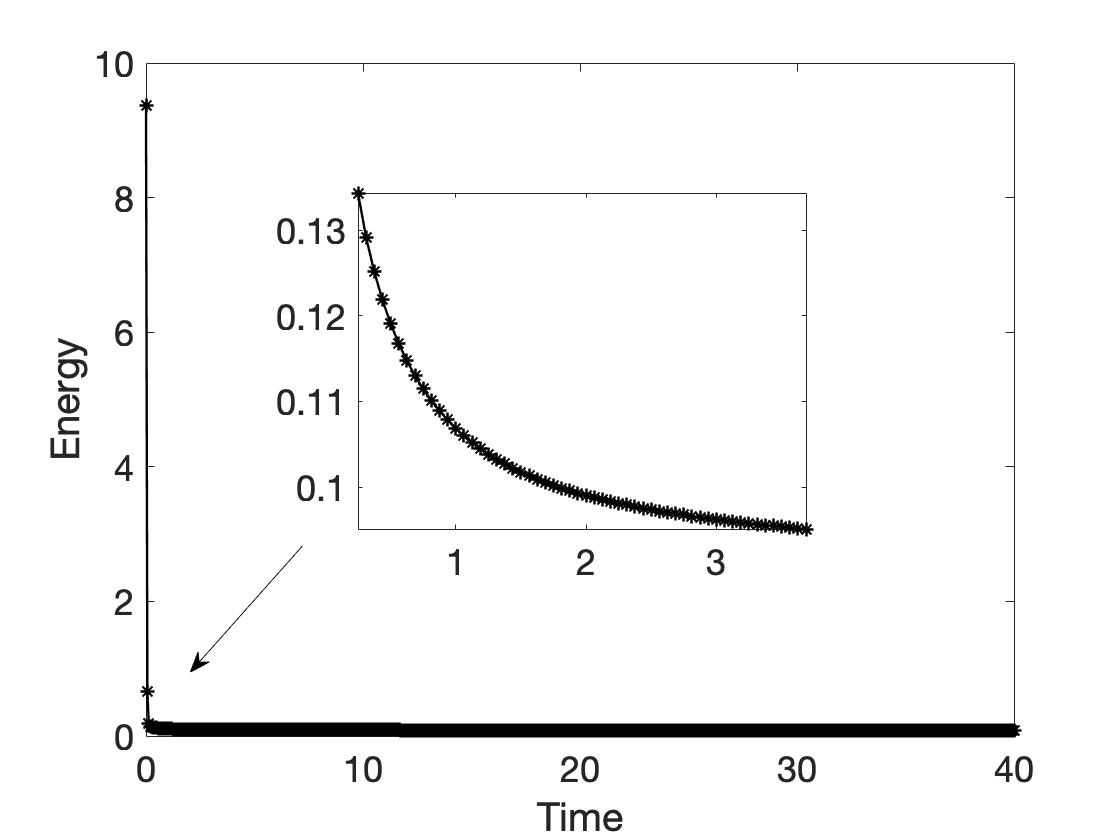} \qquad
\includegraphics[width = 1.9in,height=1.5in]{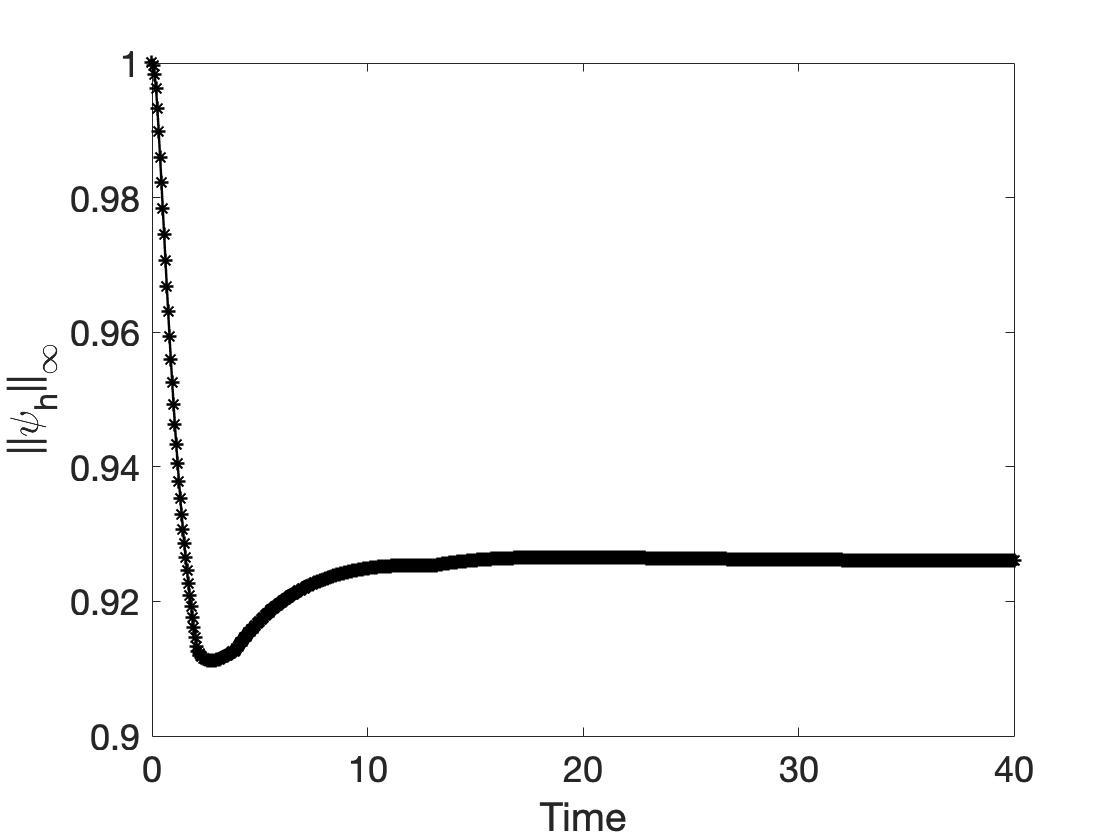}
\caption{Discrete energy and maximum bound of the discrete order parameter for Example 2.}
\label{fig:Lenergy}
\end{figure}

\begin{figure}[!ht]
	\subfloat{
		\centering
		\includegraphics[width = 1.2in]{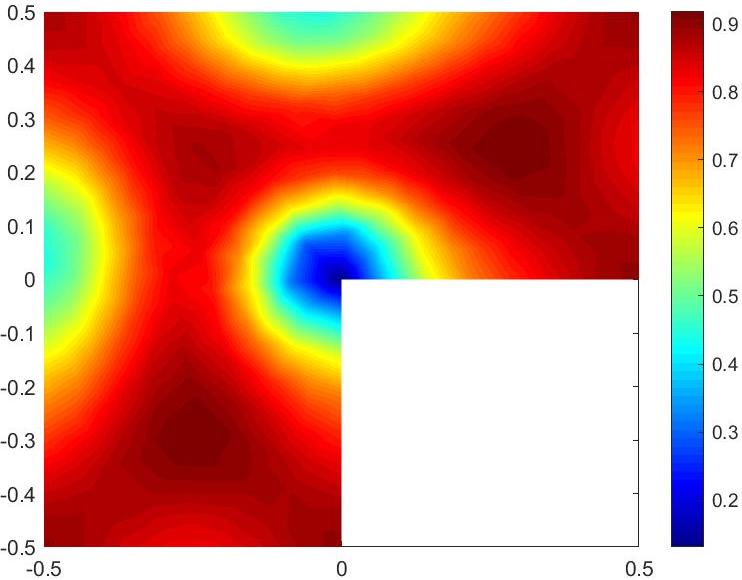}
		\includegraphics[width = 1.2in]{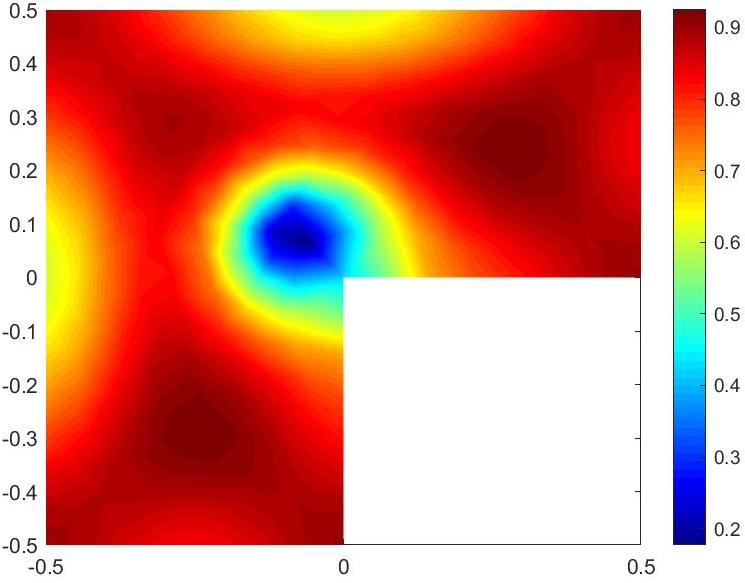}
		\includegraphics[width = 1.2in]{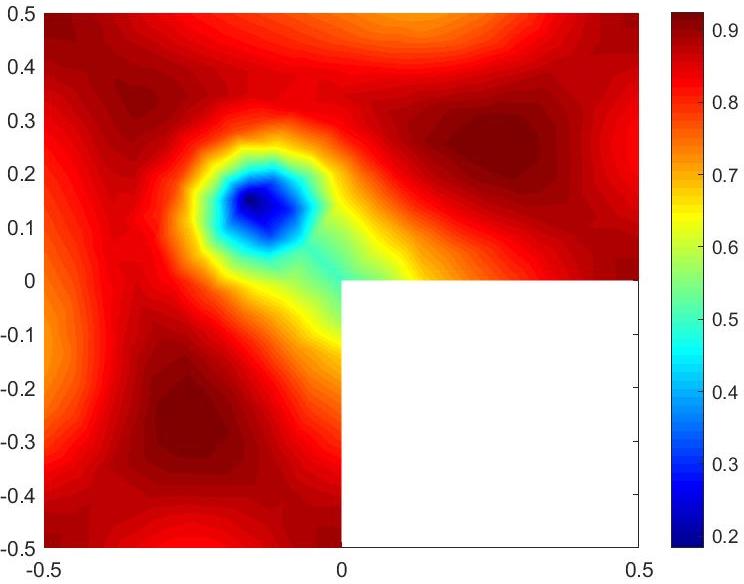}
		\includegraphics[width = 1.2in]{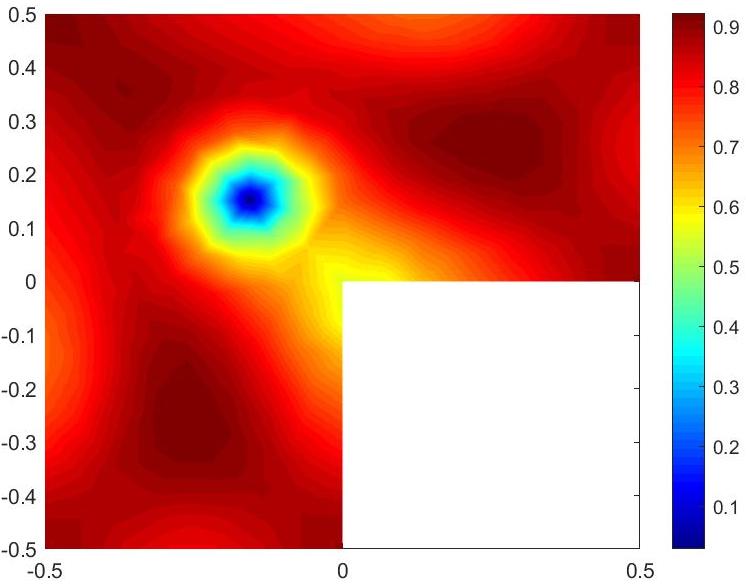}
	}

\subfloat{
	\centering
	\includegraphics[width = 1.2in]{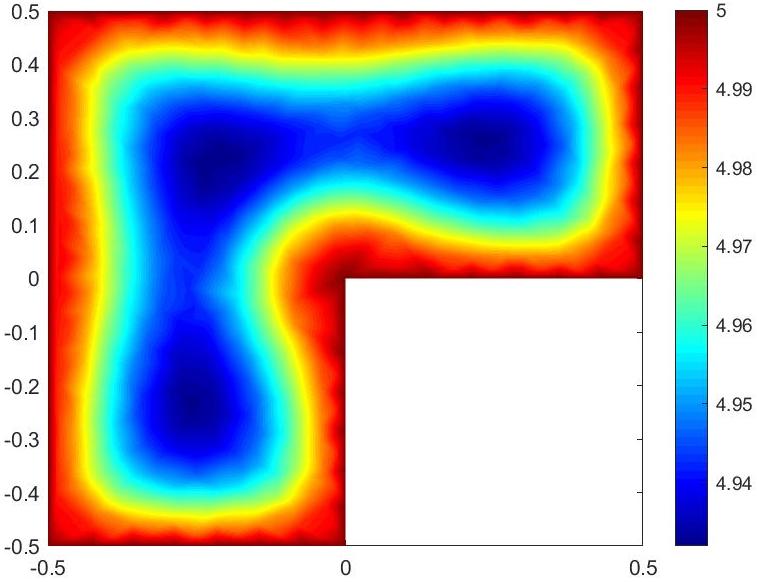}
	\includegraphics[width = 1.2in]{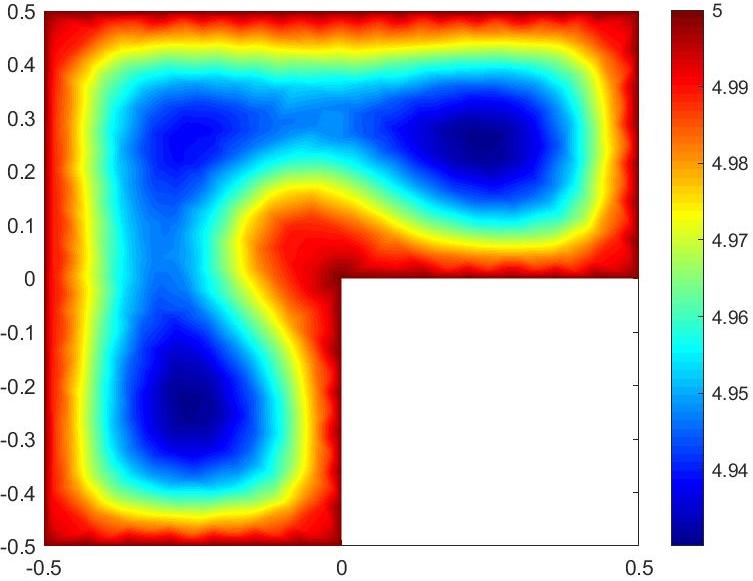}
	\includegraphics[width = 1.2in]{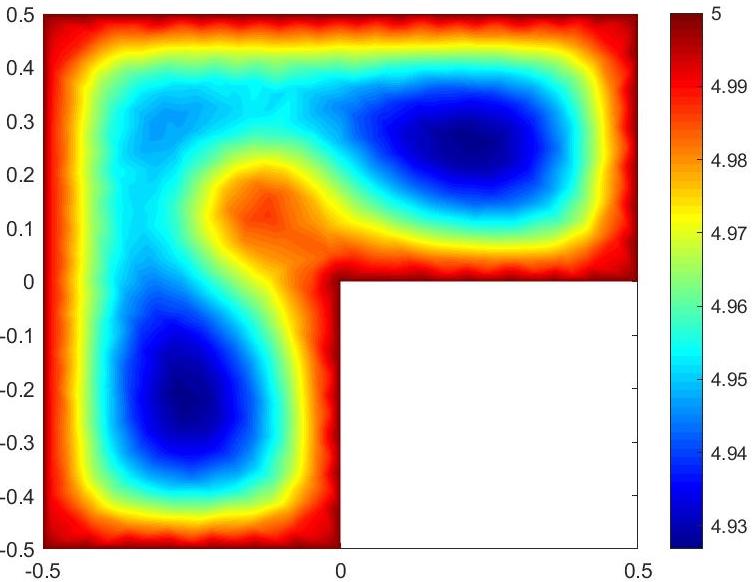}
	\includegraphics[width = 1.2in]{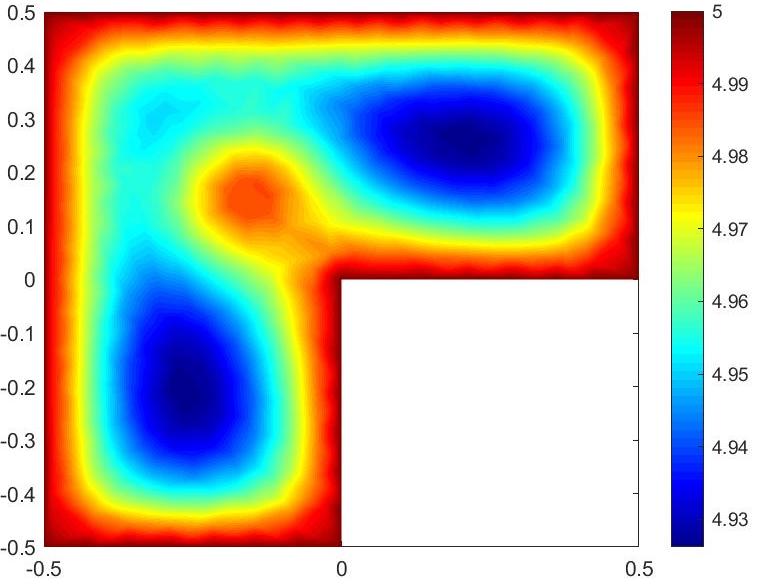}
}\\
\centering $t=5$\hspace{2cm} $t=10$\hspace{2cm}  $t=20$ \hspace{2cm}  $t=40$
\caption{$|\psi_h|$ (above) and $\nabla \times \bA_h$ (below) at $t=5$, $10$, $20$ and $40$  for Example 2.}
\label{fig:Lshape}
	\end{figure}

Fig. \ref{fig:Lshape} plots $|\psi_h|$ and $\nabla\times \bA_h$ at different times by the scheme \eqref{dis1}--\eqref{psieq}. It shows that one vortex enters the material from the reentrant corner as the time increases, which is similar to those reported in \cite{gao2017efficient,li2015new}. Physically speaking, the superconducting density should be between 0 and 1, and the average magnetic field should be less than~$\bH$ when the superconductor is in a mixed state \cite{du1992analysis}.
The numerical results in Fig. \ref{fig:Lenergy} and Fig. \ref{fig:Lshape} coincide with this physical observation. 

Comparing with the numerical schemes in \cite{gao2017efficient,gao2016new,gao2018analysis,hong2022efficient,li2020a,li2015new,li2017mathematical} where this example was tested, there are four virtues of the proposed scheme. Firstly, it is easy for the proposed scheme to implement the boundary condition, where the conventional finite element method and the second order scheme in \cite{gao2014optimal} need to deal with the extra boundary condition. Secondly, the physical boundary condition for the proposed scheme avoids the appearance of the nonphysical numerical phenomena, where the aforementioned schemes generate incorrect solutions when $M=16$ and $32$ as reported in \cite{gao2017efficient,li2015new}. Thirdly, the proposed scheme solves a decoupled linear system of two variables without introducing any auxiliary variables as in the mixed element schemes in \cite{gao2017efficient,gao2016new,li2020a,li2015new}, and the  computational cost of the linear system is smaller compared to the nonlinear systems of the numerical schemes  in \cite{hong2022efficient,li2017convergence}. Moreover, the unconditionally energy stability is guaranteed for the proposed scheme, which allows relatively larger time steps and therefore the application of adaptive time stepping strategies to speed up  simulations.

\subsection{Example 3: hollow superconductor}
We present simulations of vortex dynamics of a type-II superconductor in a square domain $[0,10]^2$ with four square holes $\{(x,y): x\ \mbox{and}\ y\in [2,3]\cup [7,8]\}$. We set
$\sigma=1$, $\kappa =4$, $\psi^0=1.0$, $\boldsymbol{A}^0=(0,0),$
and test on two different external magnetic fields $\bH=1.1$ and $1.9$ { with $\mu_n=2$}.  The example was tested before in \cite{gao2016new,hong2022efficient,peng2014vortex}. We simulate the motion on triangulations generated by Gmsh \cite{geuzaine2009gmsh}. Since the discrete energy decays as proved in Theorem \ref{th:energy}, we adopt the adaptive time-stepping strategy in \cite{qiao2011adaptive} which takes the form
\begin{equation}\label{adptiveT}
\tau^n=\max\{\tau_{\rm min}, \frac{\tau_{\rm max}}{\sqrt{1+ \alpha | \frac{G_h^{n-1} - G_h^{n-2}}{\tau^{n-1}}|^2}}\},
\end{equation}
where the positive constant $\alpha=10^5$, $\tau_{\rm max}=0.2$ and $\tau_{\rm min}=0.02$.

\begin{figure}[!ht]
	\centering
\includegraphics[width = 1.7in]{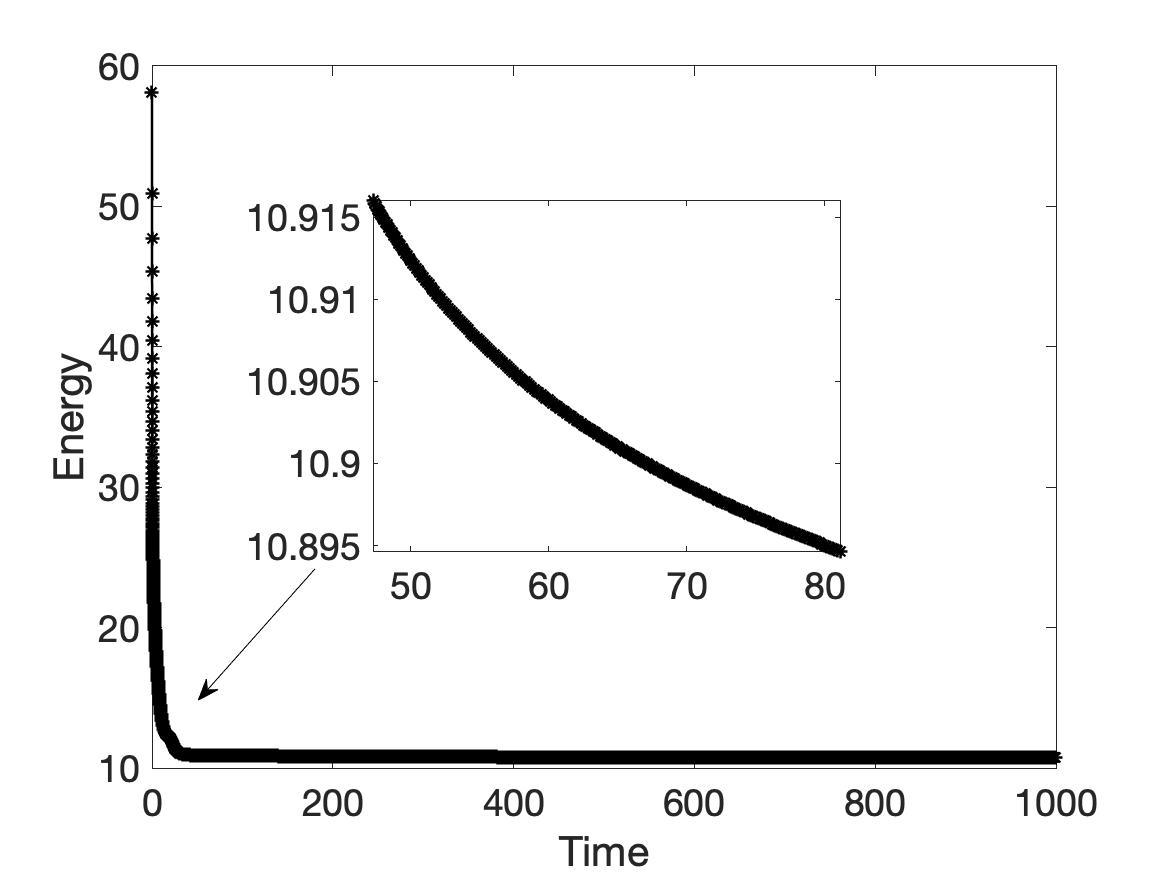} \qquad
\includegraphics[width = 1.7in]{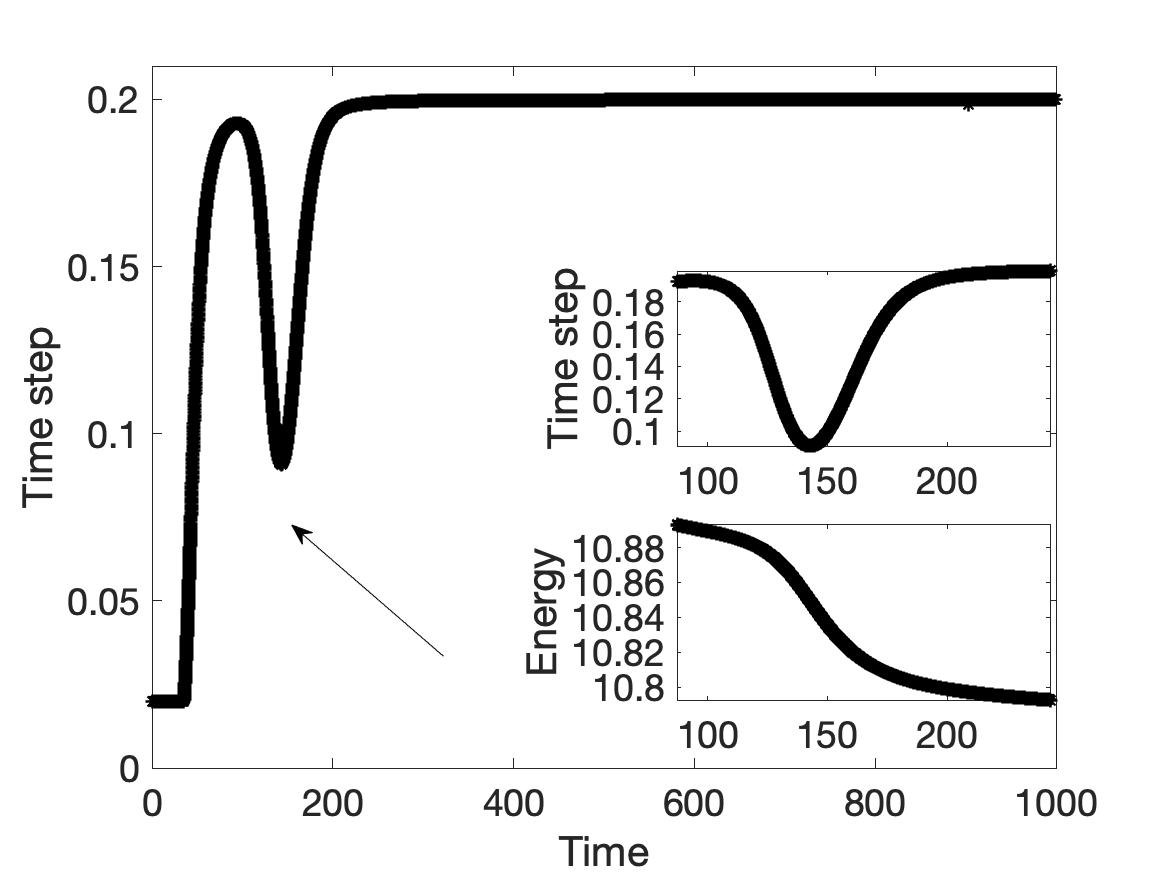}
\caption{Discrete energy and time steps for Example 3 with $\bH=1.1$.}
\label{fig:Holeenergy11}
\end{figure}

\begin{figure}[!ht]
	\centering
\includegraphics[width = 1.8in]{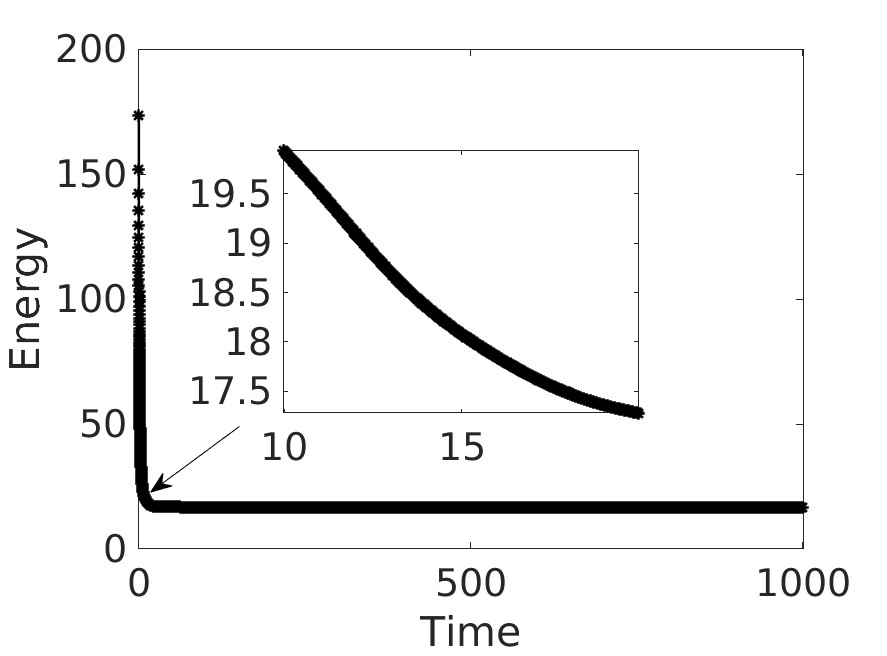} \qquad
\includegraphics[width = 1.8in]{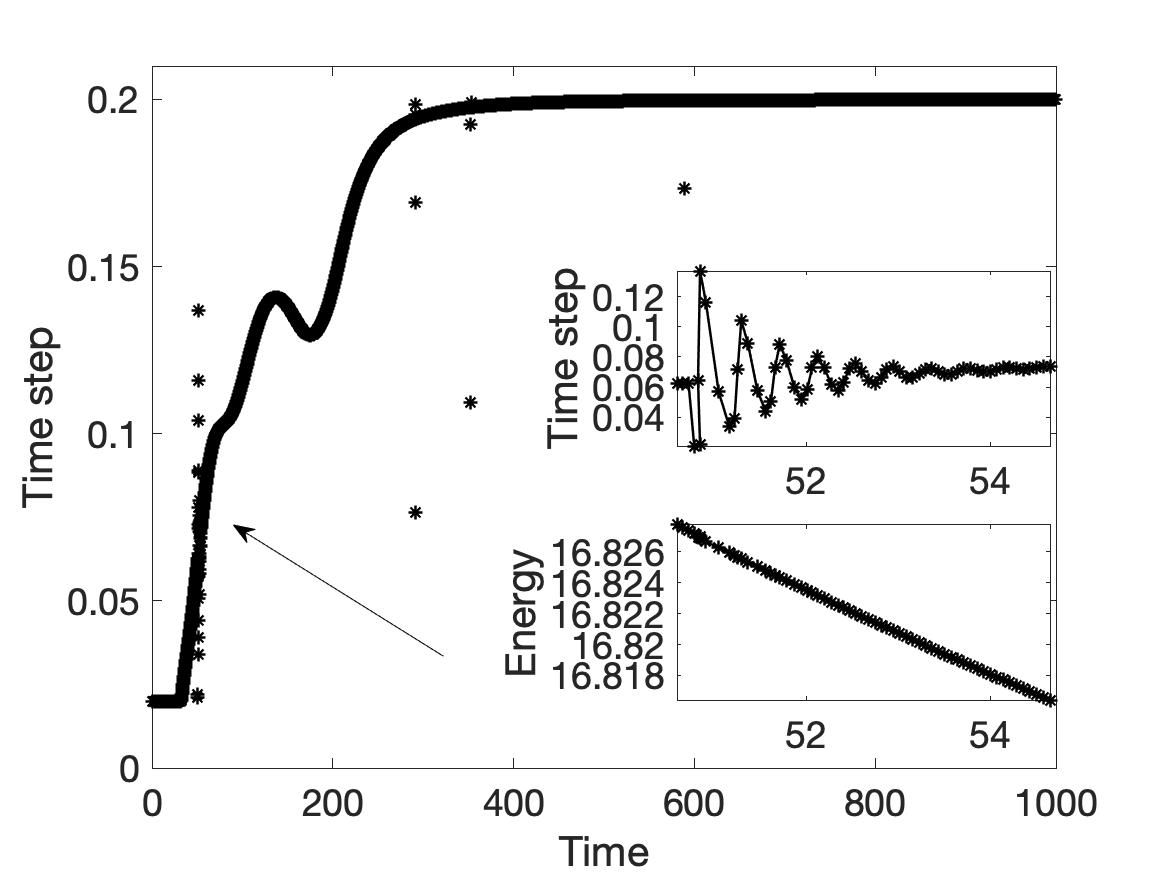}
\caption{Discrete energy and time steps for Example 3 with $\bH=1.9$.}
\label{fig:Holeenergy19}
\end{figure}

\begin{figure}[!ht]
	\centering
\includegraphics[width = 1.8in]{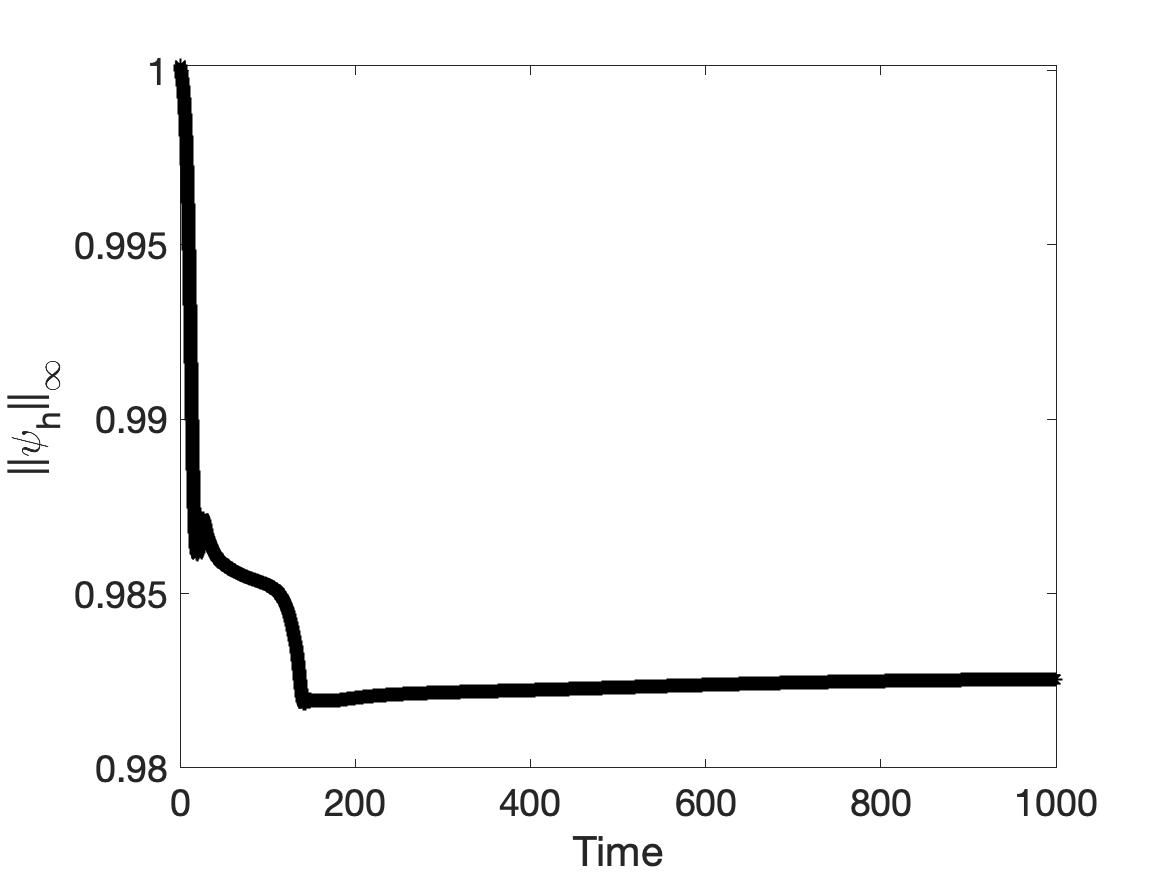} \qquad
\includegraphics[width = 1.8in]{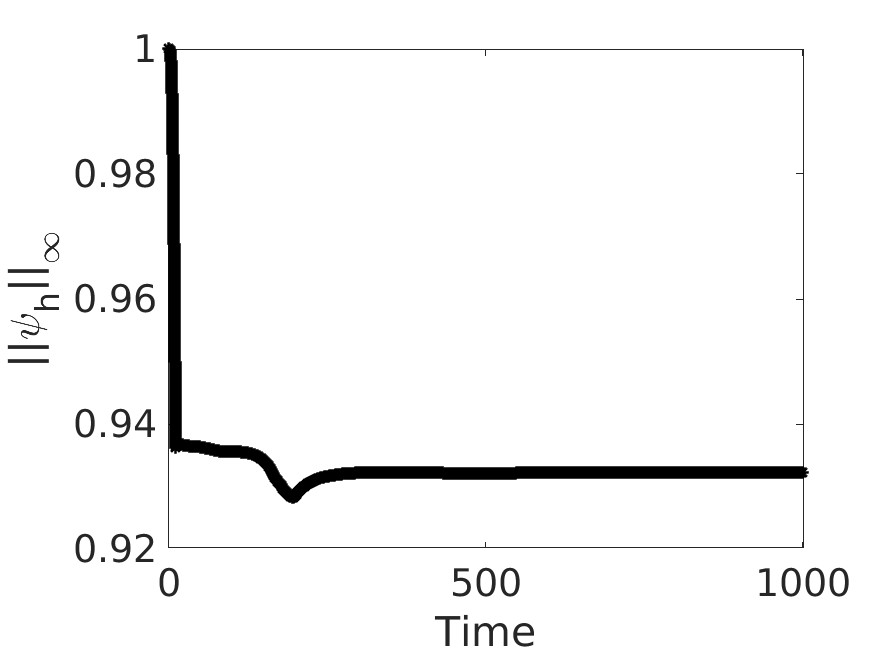}
\caption{Discrete maximum bound of $|\psi_h|$ for Example 3 with $\bH=1.1$ (left) and $\bH=1.9$ (right).}
\label{fig:MBP}
\end{figure}

Fig. \ref{fig:Holeenergy11} and Fig. \ref{fig:Holeenergy19} plot the discrete energy and time steps of the proposed scheme with $\bH=1.1$ and $\bH=1.9$ when $t\le 1000$, respectively. As shown in Fig. \ref{fig:Holeenergy11}, the adaptive time-stepping strategy can successfully capture the change of discrete energy and save computational time.
Note that the time steps are nearly $\tau_{\rm max}=0.2$ when $t\ge 100$ for $\bH=1.1$, which is much larger than $\tau=0.005$  and $\tau=0.02$ in \cite{gao2016new} and \cite{hong2022efficient}, respectively. When the applied magnetic field $\bH=1.9$, the new approach gives a physical simulation of the vortex motion until $t=1000$  with the time step nearly $\tau_{\rm max}=0.2$ when $t\ge 400$ as shown in Fig. \ref{fig:Holeenergy19}.
The vortex motion under $\bH=1.9$ was simulated for $t\le 150$ in \cite{gao2016new} with time step $\tau=0.002$ on a triangulation with $405416$ elements. A nonphysical phenomenon starts to appear in the simulation when $t=10$. We use the proposed scheme \eqref{dis1}-\eqref{psieq} with the adaptive time-stepping strategy \eqref{adptiveT} on a triangulation with $516526$ elements and the simulation exhibits physical phenomenon before $t=800$ and nonphysical behavior starts to appear after $t=800$. As shown in Fig. \ref{fig:Holeenergy19}, Fig. \ref{fig:MBP} and Fig. \ref{fig:Hole19}, our approach on a triangulation with $786482$ elements gives a physical simulation of the vortex motion under $\bH=1.9$ until $t=1000$  with the time step nearly $\tau_{\rm max}=0.2$ when $t\ge 400$. This implies that the proposed scheme \eqref{dis1}-\eqref{psieq} with adaptive time-stepping strategy is much more stable and efficient in long-time simulations. As shown in Fig. \ref{fig:Holeenergy19}, the discrete energy decays even when the time step is not changing continuously which also verifies the unconditional energy decay property of the proposed numerical scheme.

\begin{figure}[!ht]
	\subfloat{
		\centering
	\includegraphics[width = 1.1in]{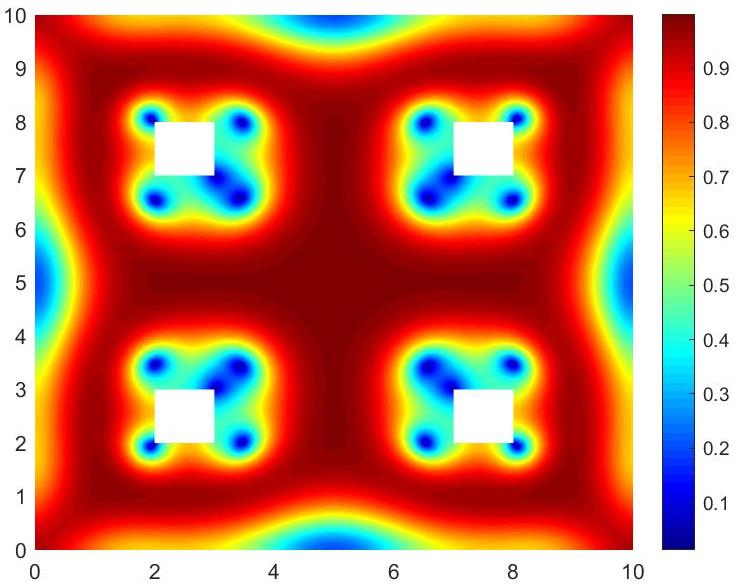}
	\includegraphics[width = 1.1in]{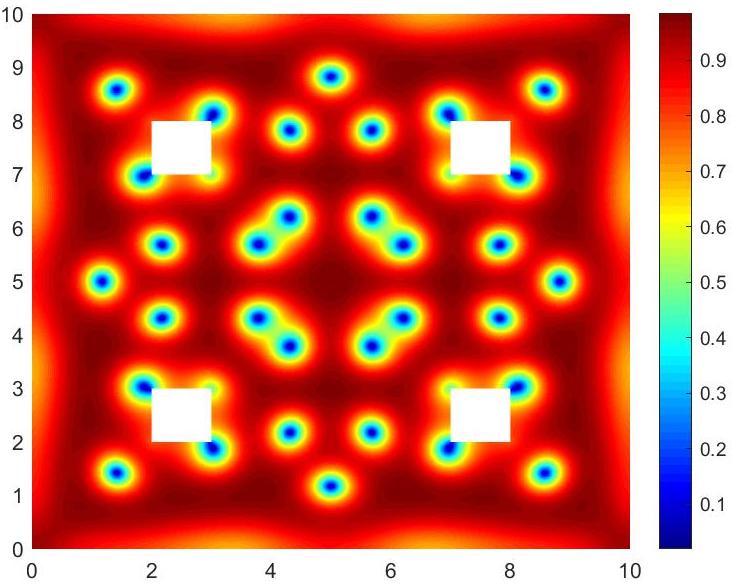}
	\includegraphics[width = 1.1in]{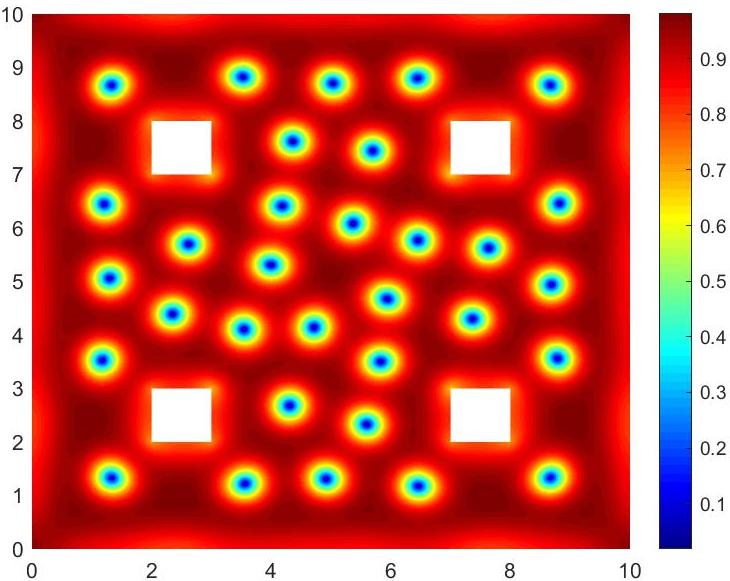}
	\includegraphics[width = 1.1in]{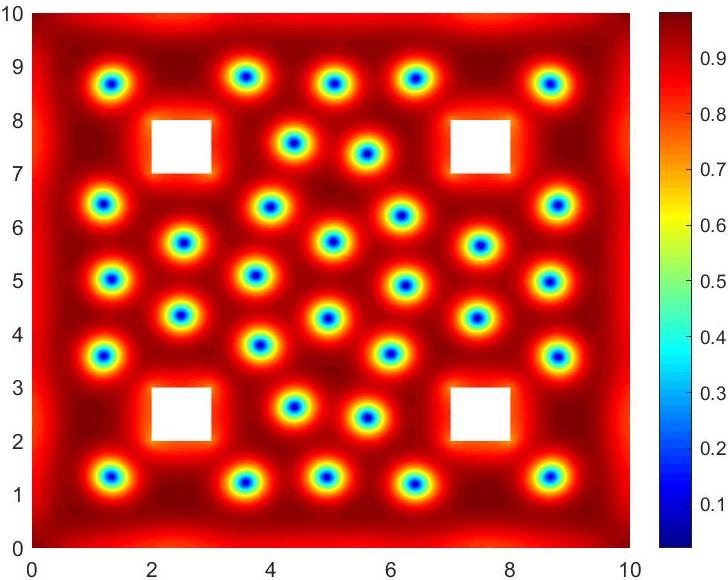}
	}\\
	\subfloat{
		\centering
	\includegraphics[width = 1.1in]{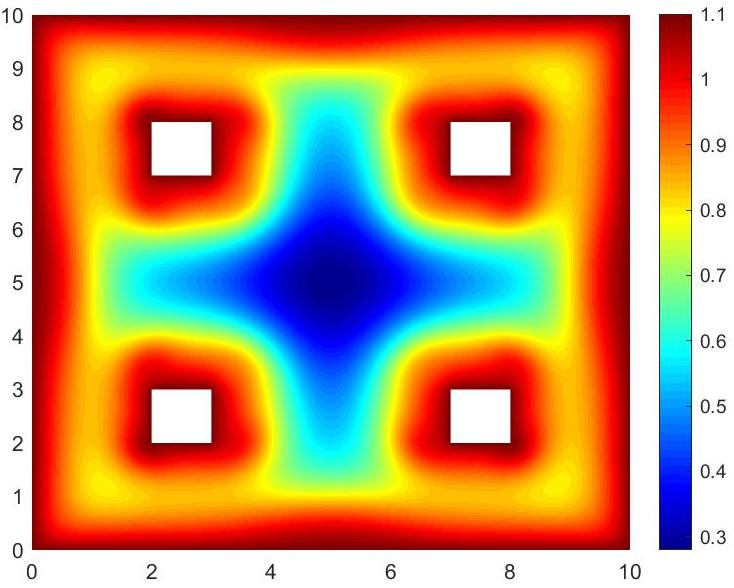}
	\includegraphics[width = 1.1in]{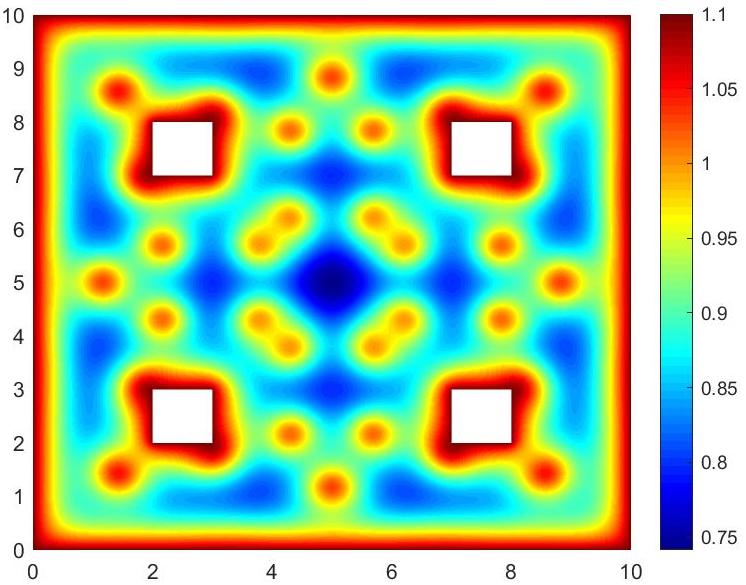}
	\includegraphics[width = 1.1in]{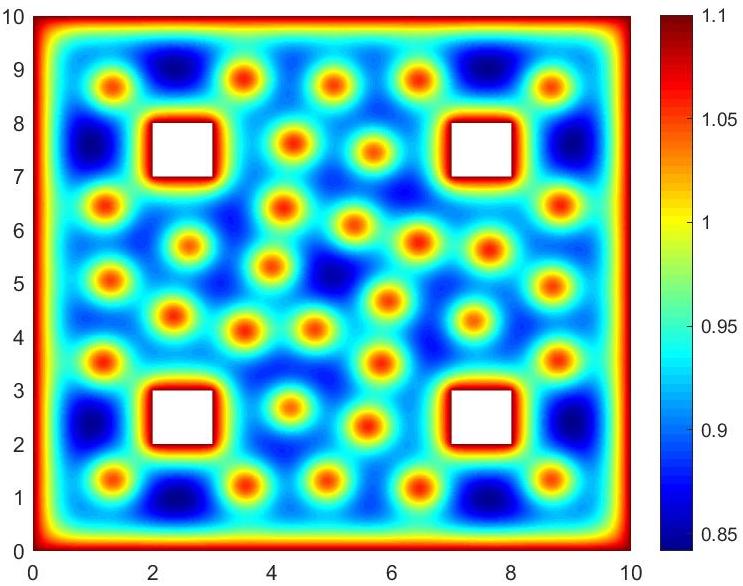}
	\includegraphics[width = 1.1in]{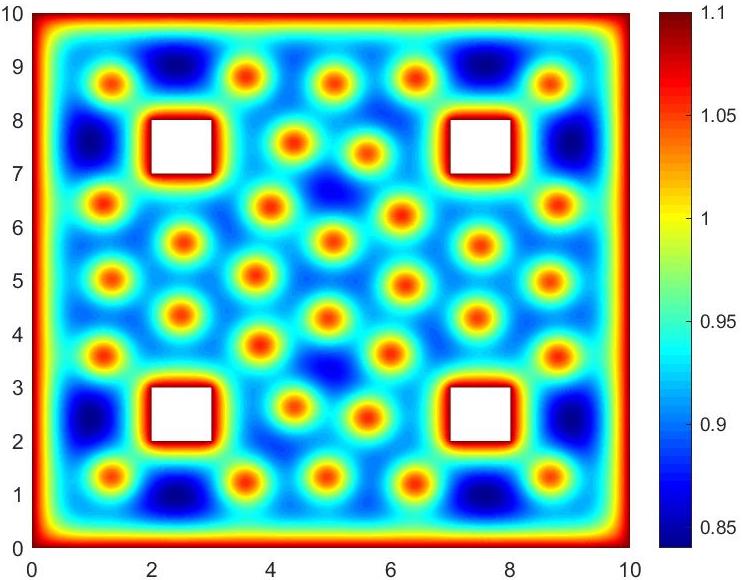}
	}
\\
\centering \hspace{-0.2cm}$t=10$\hspace{1.7cm} $t=50$\hspace{1.7cm}  $t=200$ \hspace{1.6cm}  $t=1000$
\caption{$|\psi_h|$ (above)  and $\nabla\times \bA_h$ (below) at $t=10$, 50, 200, 1000 for Example 3 with $\bH=1.1$ on a triangulation with 516526 elements.}
\label{fig:Hole11}
	\end{figure}

\begin{figure}[!ht]
	\subfloat{
		\centering
	\includegraphics[width = 1.1in]{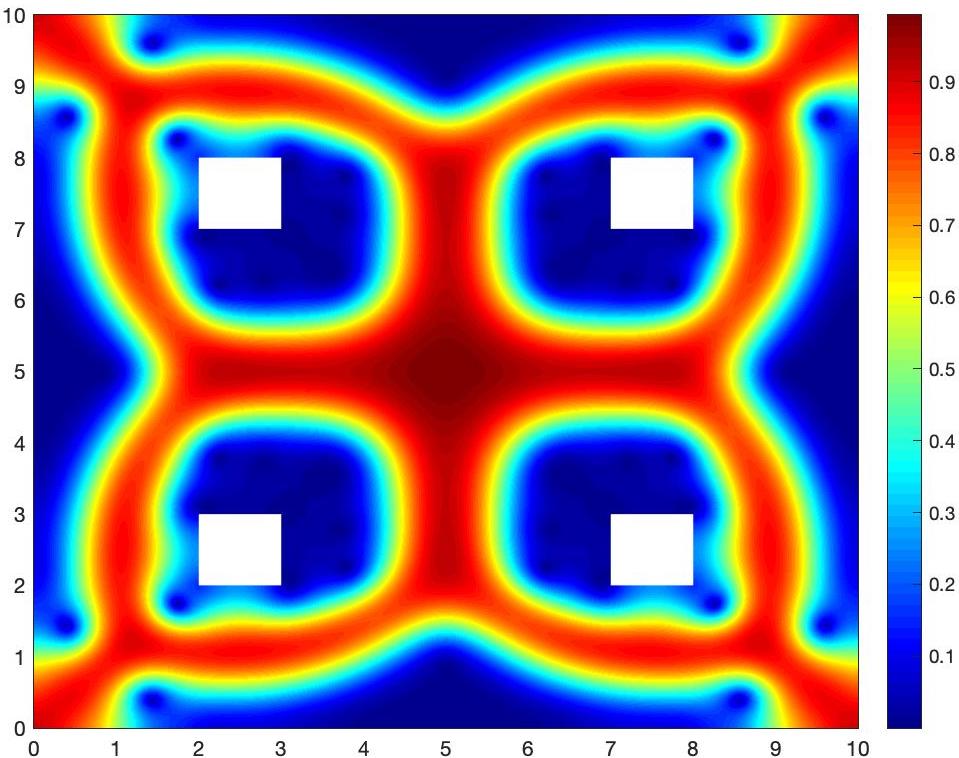}
	\includegraphics[width = 1.1in]{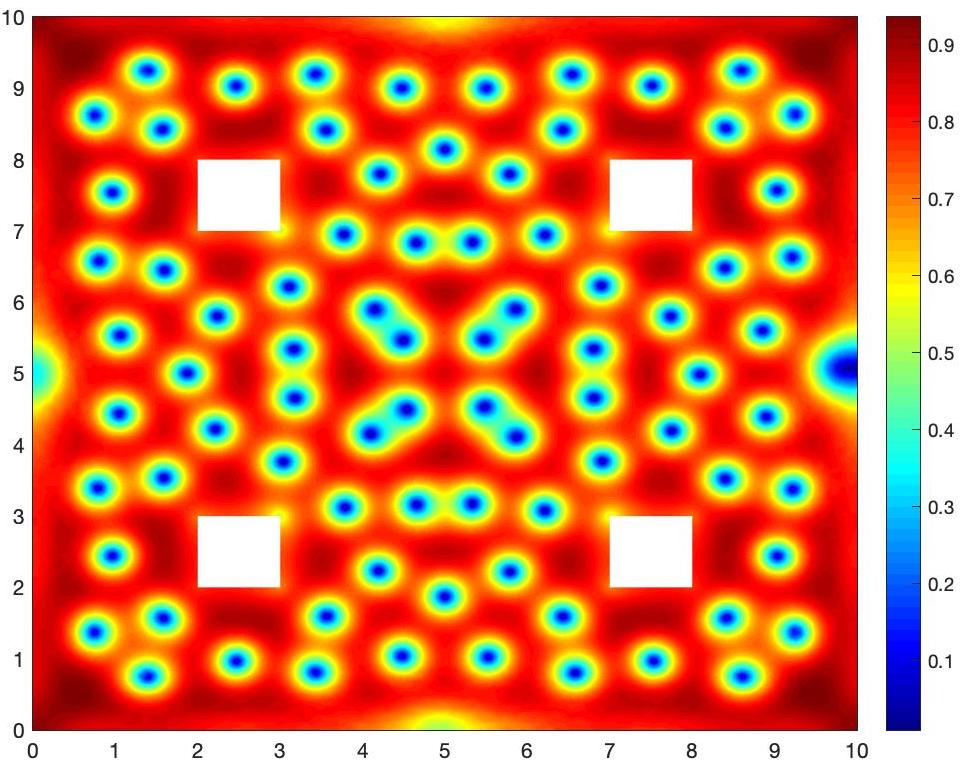}
	\includegraphics[width = 1.1in]{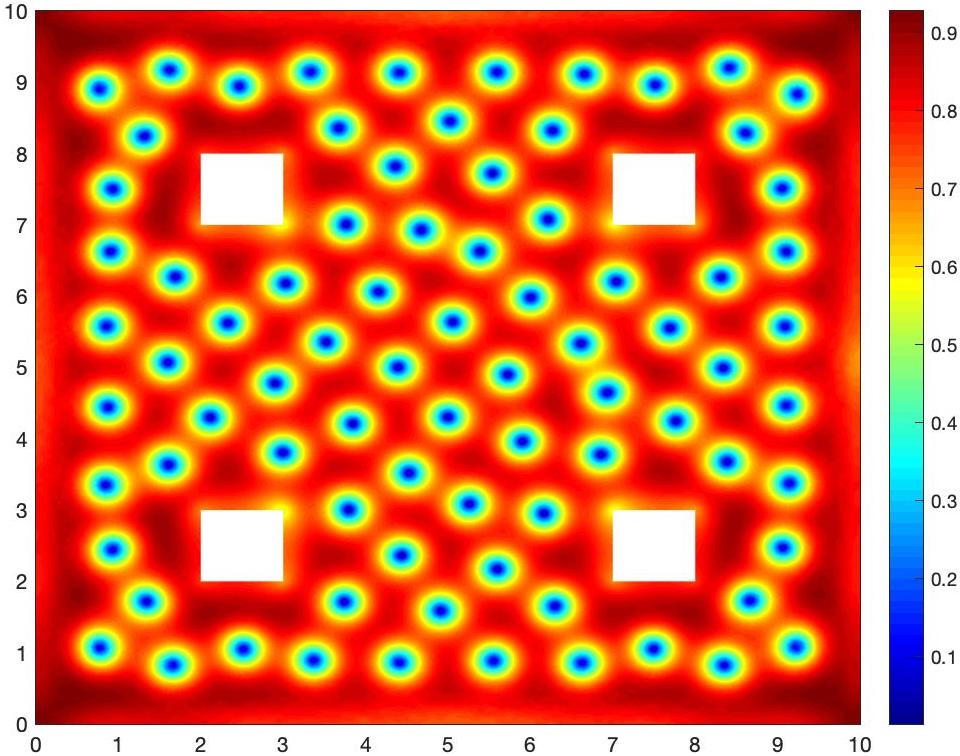}
	\includegraphics[width = 1.1in]{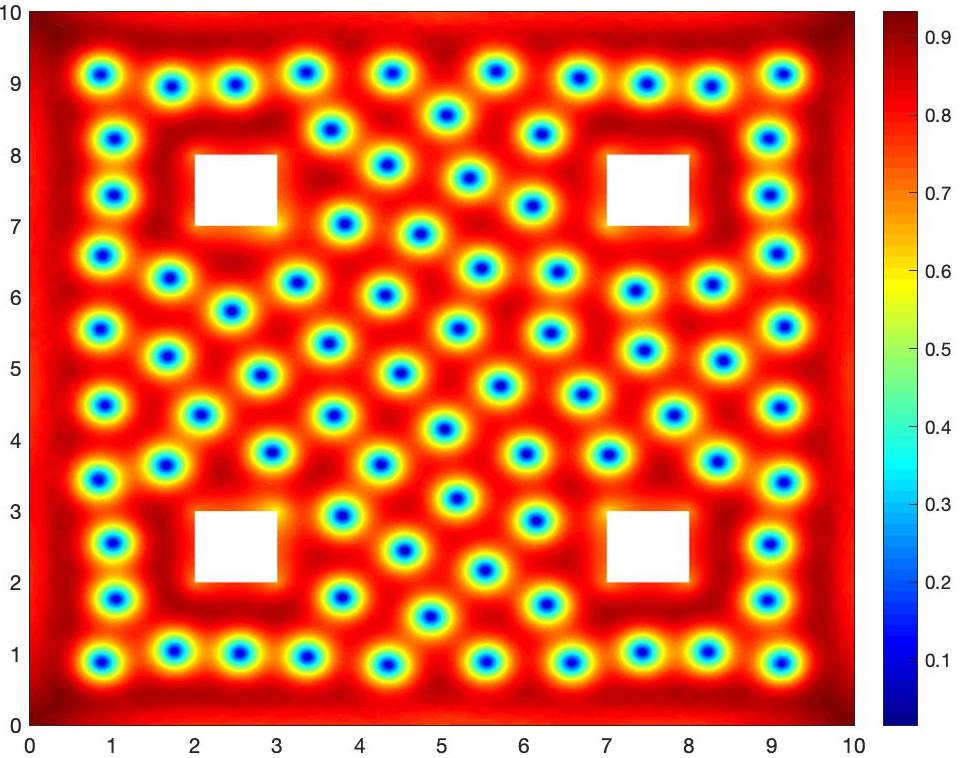}
	}
	\\
	\subfloat{
		\centering
	\includegraphics[width = 1.1in]{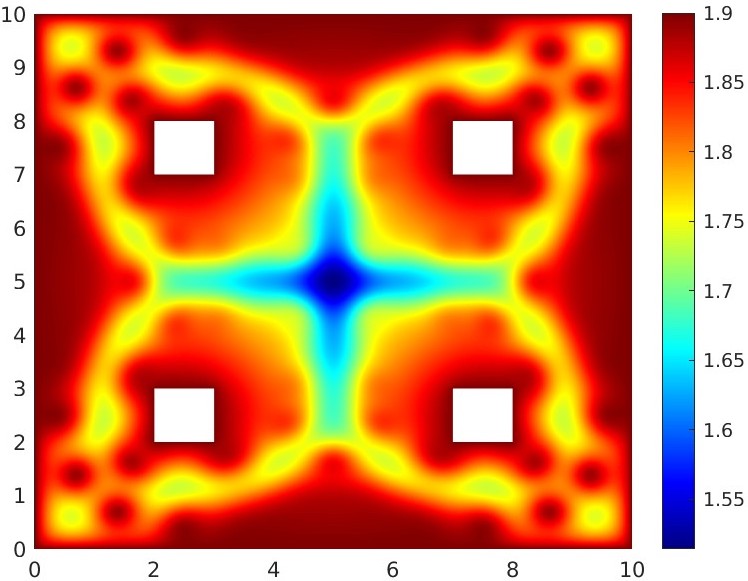}
	\includegraphics[width = 1.1in]{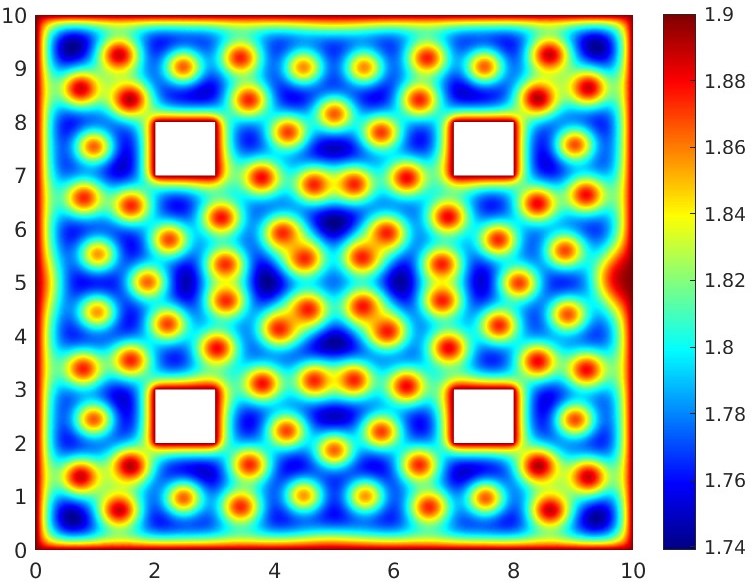}
	\includegraphics[width = 1.1in]{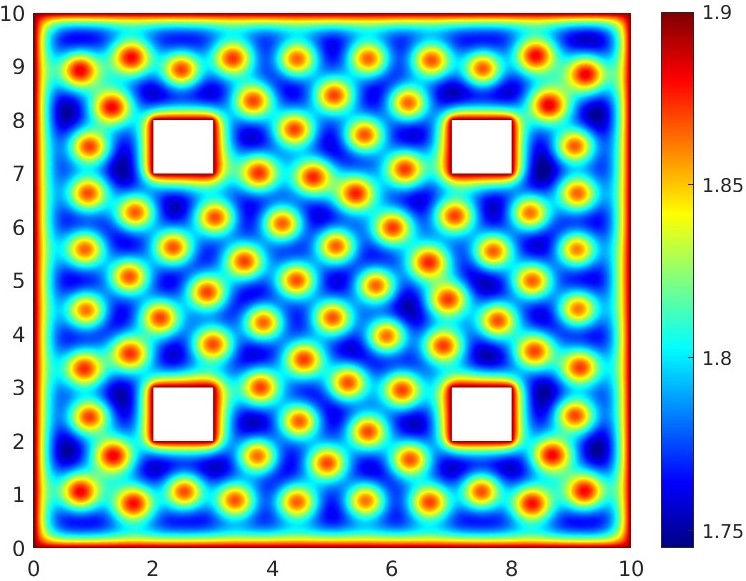}
	\includegraphics[width = 1.1in]{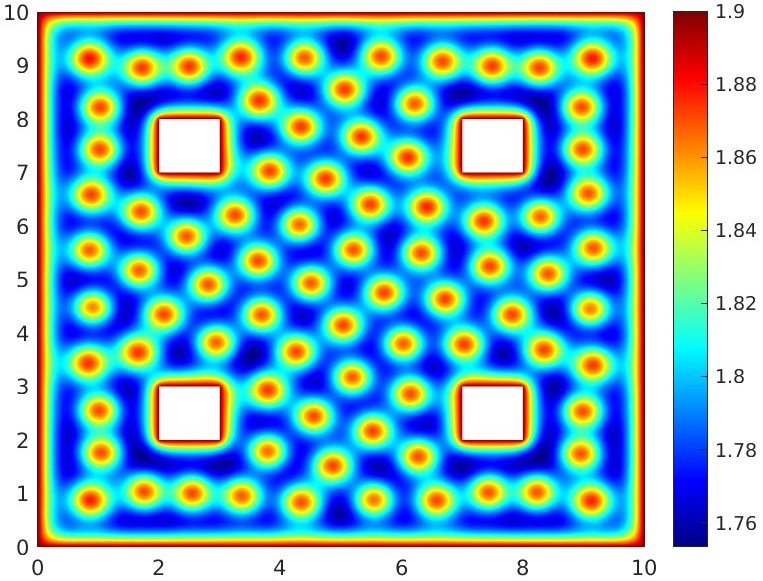}
	}
\\
\centering \hspace{-0.2cm}$t=10$\hspace{1.7cm} $t=50$\hspace{1.7cm}  $t=200$ \hspace{1.6cm}  $t=1000$
\caption{$|\psi_h|$ (above)  and $\nabla\times \bA_h$ (below) at $t=10$, 50, 200, 1000 for Example 3 with $\bH=1.9$ on a triangulation with 786482 elements.}
\label{fig:Hole19}
	\end{figure}

Fig. \ref{fig:Hole11} and Fig. \ref{fig:Hole19} plot $|\psi_h|$ and $\nabla\times \bA_h$ at $t=10$, $50$, $200$ and $1000$ for $\bH=1.1$ and $\bH=1.9$, respectively.
As observed in Fig. \ref{fig:Hole11} and Fig. \ref{fig:Hole19}, the vortices start to penetrate the material near the four square holes. When $\bH$ becomes larger, more vortices are generated and triangulation with a much smaller mesh size is required to resolve the singularity of solutions, which coincides with the physical phenomenon.
Physically speaking, the penetrated magnetic flux will separate into the smallest bundle to guarantee the largest interface area since the interface energy in type-II superconductors is negative, and the vortices form a lattice because of the weak repulsive interactions among them. In long-time simulations, numerical schemes with high convergence accuracy may produce some nonphysical numerical phenomenon because of the lack of stability. This nonphysical phenomenon often happens near the reentrant corners when the applied magnetic field is strong. The vortex dynamics in Fig. \ref{fig:Hole11} and Fig. \ref{fig:Hole19} show that the proposed numerical scheme is robust and stable even when $\bH=1.9$.

\section{Conclusions}
\label{sec:conclusions}
In this paper, we propose a decoupled scheme for the TDGL equations under the temporal gauge by combining the ETD method and the backward Euler method for time discretization and finite element methods for spatial discretization.
Compared to the existing schemes for the TDGL equations, the proposed numerical scheme admits four advantages. Firstly, the scheme and all the energy stability analysis, MBP analysis and error estimate work for superconductors with complicated shapes. Secondly, an unconditional energy dissipation law is proved for the proposed scheme. This allows the application of an adaptive time-stepping strategy which can significantly speed up simulations compared to other numerical schemes for the TDGL equations in the literature using a fixed time step. Thirdly, the discrete MBP is proved for the order parameter which indicates the stability of the numerical scheme, while no other numerical schemes using finite element methods can preserve the MBP property theoretically. The analyzing technique can also be used in other problems with complex order parameters. Finally, the relatively low regularity of the numerical solutions prevents the appearance of some nonphysical numerical solutions.

For the discrete scheme in Remark \ref{remark:ppr} with gradient recovery techniques, the discrete MBP is also guaranteed under the mesh requirements in Assumption \ref{ass:mesh}. But how to preserve the energy dissipation law in a discrete sense is still an open problem. A major difficulty comes from the discretization of the coupling nonlinear terms in the equations for both the magnetic field and the order parameter. The proposed scheme \eqref{dis1}-\eqref{psieq} is only of first order in time. The fact that the differential operator $L[\bA]$ depends on the variable $\bA$ leads to the failure in constructing high order MBP-preserving numerical schemes using the standard ETD methods with second order accuracy. How to design an MBP-preserving numerical scheme with higher accuracy in time is also open, which requires some delicate treatment with respect to the coupling terms of the TDGL equations. A fast solver of numerical schemes is important in simulating the vortex motion of superconductors, especially when the shape of the superconductor is not smooth and a strong external magnetic field is applied. The design of fast solvers for the proposed numerical scheme and the theoretical analysis to guarantee the efficiency of the solver deserve deeper study.

\bibliographystyle{plain}
\bibliography{bibsuperConduct}

\end{document}